\documentclass[11pt]{article}

\usepackage{inputenc}%[latin1]
\usepackage{amsmath}
\usepackage{amsfonts}
\usepackage{amssymb}
\usepackage{graphicx}
\usepackage{graphicx}
\usepackage{subcaption}
\usepackage[margin=1in]{geometry}
\usepackage[T1]{fontenc}
\usepackage[english]{babel}
\usepackage{mathrsfs}
\usepackage{amsthm}
\usepackage[all]{xy}
\usepackage{geometry}
\usepackage{enumerate}
\usepackage{enumitem}
\allowdisplaybreaks
\usepackage{color}
\usepackage{appendix}
\usepackage{authblk}
\usepackage{tikz}
\usepackage{hyperref}

\title{Rolling carpet strategy to reduce mosquito populations in two-dimensional space}

\author[1]{Lu\'{\i}s Almeida}
\author[2]{Alexis L\'{e}culier}
\author[3]{Nga Nguyen}
\author[3]{Nicolas Vauchelet}

\affil[1]{LPSM - Sorbonne Universit\'{e} and Universit\'{e} Paris-Cit\'{e}, France \texttt{(luis.almeida@cnrs.fr)}}
\affil[2]{Universit\'e de Bordeaux - Institut de Mathématiques de Bordeaux, UMR 5251, F-33405, Talence, France \texttt{(alexis.leculier@u-bordeaux.fr)}}
\affil[3]{Université Sorbonne Paris Nord, Laboratoire Analyse, Géométrie et Applications, LAGA, CNRS UMR 7539,
F-93430, Villetaneuse, France \texttt{(vauchelet@math.univ-paris13.fr)}}
\begin{document}                     
\maketitle

\begin{abstract}
Mosquitoes are vectors of numerous diseases; a strategy to fight the spread of these diseases is to control the vector population. In this article, we focus on the use of the sterile insect technique. Starting from a reaction-diffusion system, we show the existence of 'forced' traveling waves obtained by translating the intervention zone at constant speed. This result is proved in a two-dimensional space by using the radial symmetry. 

\end{abstract}
\textbf{Keywords~:} Reaction-diffusion system, traveling waves, population dynamics. 
\\
\textbf{AMS Subject Classification~:} 35K57, 92D25, 35C07.  \\

%\tableofcontents

%%%%%%%%%%%%%%%%%%%%%%%%%%%%%%%%%%%%%
\newtheorem{theorem}{Theorem}
\newtheorem{corollary}{Corollary}
\newtheorem{lemma}{Lemma}
\newtheorem{definition}{Definition}
\newtheorem{proposition}{Proposition}
\newtheorem{example}{Example}
\newtheorem{definition-proposition}{Definition-Proposition}
\newtheorem*{notation}{Notation}
\newtheorem*{remark}{Remark}

%%%%%%%%%%%%%%%%%%%%%%%%%%%%%%%%%%%%%%%%%%%%
\newcommand{\RR}{\mathbb{R}}
\newcommand{\barE}{\overline{E}}
\newcommand{\barM}{\overline{M}}
\newcommand{\barF}{\overline{F}}
%%%%%%%%%%%%%%%%%%%%%%%%%%%%%%%%%%%%%%%%%%%% 

\section{Introduction}

Many species of mosquitoes are vectors for numerous diseases, for instance \textit{Aedes} mosquitoes are vectors for chikungunya, zika and dengue. Without an efficient vaccine, reducing the vector population remains the key to controlling the spread of such diseases. The sterile insect technique (SIT) and the closely related Incompatible Insect Technique (IIT) aim to reduce the size of the insect population by releasing massively sterile males (for SIT, or incompatible males in the case of IIT). Although this technique was introduced to eradicate other insect species (see e.g. \cite{Enkerlin2017}), it has recently been successfully implemented in the field to control mosquito populations \cite{Bellini2013,CAP, GAT, ZHE}. In addition, many field trials of SIT are underway; see the review article \cite{Bouyer2024} where the issue of the scalability of the SIT for mosquitoes is also addressed.
Although ongoing efforts have reduced the cost of producing and releasing sterilized males \cite{Maiga}, the implementation of this strategy in large areas remains a challenge.

In this paper, we investigate from a mathematical point of view a general strategy to extend the SIT in a large spatial domain. This strategy consists in moving the release region to extend the free mosquito area and is called a \textit{"rolling carpet"} strategy.
A numerical investigation of this idea had been proposed in \cite{SEI} (see also \cite{ANG}).
Recently, the mathematical analysis of this technique has been considered for a simple scalar case in one dimension in \cite{ALM3}, and its optimization in \cite{bressan, AlmLecNadPriv}. A more general system has also been considered in \cite{LEC2023}.
The main objective of this paper is to generalize the mathematical analysis of the \textit{"rolling carpet"} strategy for a complete system of mosquito dynamics and in the two-dimensional case.

More precisely, we consider the following system that models the dynamics of a mosquito population with several stages~: $E$ density of the aquatic phase, $F$ density of fertilized females, $M$ density of males, $M_s$ density of sterilized males.
These quantities depend on the time variable $t>0$ and on the space variable $x\in\RR^2$. The following mathematical system governing the dynamics of these quantities was proposed in \cite{STR19} without spatial diffusion (see e.g. \cite{ANG,ALM1} for its natural extension incorporating the spatial dependency),
\begin{subequations}\label{syst}
\begin{align}\label{syst:E}
&\partial_t E = b F \left(1 - \frac{E}{K}\right) - (\mu_E+\nu_E) E \\
\label{syst:M}
&\partial_t M - D \Delta M = (1-\rho) \nu_E E - \mu_M M  \\
\label{syst:F}
&\partial_t F - D \Delta F = \rho \nu_E E \frac{M}{M+\gamma_s M_s}\Gamma(M+\gamma_s M_s) - \mu_F F  \\
\label{syst:Ms}
&\partial_t M_s - D \Delta M_s = \Lambda  - \mu_s M_s 
\end{align}
\end{subequations}
In this model, $b$ is the birth rate (oviposition rate), $\mu_E$, $\mu_M$, $\mu_F$ and $\mu_s$ are the death rates for the aquatic phase, the males, the fertilized females, and the sterile males, respectively. The carrying capacity is denoted $K$, the emergence rate $\nu_E$, the sex ratio $\rho$ and  the diffusion coefficient $D$. Since we will consider the setting $D$ is constant, up to a rescaling, and to simplify the computations and notations, we will always assume in the following that $D=1$.The release function of sterilized males is denoted $\Lambda(t,x)$.
The quantity $\frac{M}{M+\gamma_s M_s}\Gamma(M)$ models the probability that a female mates with a fertile male. The parameter $\gamma_s$ models the competitiveness of sterile males, and several choices for the function $\Gamma(M)$ have been proposed in \cite{AgboAlmeidaCoron2,STR19} to model the difficulty in finding a partner when the density is low; such function has been taken to introduce an Allee effect which stabilizes the extinction equilibrium.
In this work we will consider the two following choices for the function $\Gamma$~:
\begin{equation}\label{eq:Gamma}
\Gamma(M) = 1  \ \text{ (monostable case) }, \qquad \Gamma(M) = 1-e^{-\gamma M} \ \text{ (bistable case).}
\end{equation}
Obviously in the monostable case, there is no Allee effect. We mention that another possible modeling choice to include Allee effect has also been proposed in \cite{MULTERER2019}. 
System \eqref{syst} is complemented with some given initial data in $L^\infty(\RR^2)$
$$
E(x,t=0) = E^0(x),\quad M(x,t=0) = M^0(x), \quad F(x,t=0) = F^0(x), \quad M_s(x,t=0) = M_s^0(x).
$$

The function $\Lambda$ is the release function and from a mathematical point of view it may be seen as a control function. 
When spatial dependency is neglected the question of controlling the dynamics of the population of mosquitoes thanks to the release of sterile males has been addressed by many authors. For instance in \cite{BIDI2025,AgboAlmeidaCoron2,bidiJOTA2025, bliman2024,cristofaro2024}, the stabilization by feedback control of such a system has been studied. 
Optimizing the release function, when the spatial dependency is neglected, has also attracted the attention of several mathematicians, see e.g. \cite{THOME2010,ADPV_MBE2019,ADPV2022,BLI}. A few recent works also consider the control of mosquito populations in spatially dependent settings (see \cite{ABPV2025} for the Wolbacchia case and \cite{AgboAlmeidaCoron2} for the SIT setting).

To implement the \textit{"rolling carpet"} strategy, the place where the sterilized males are released is moved at a constant speed in order to generate a "forced" wave of eradication of the mosquito population. 
Then, system \eqref{syst} is considered with $\Lambda(t,x) = \mathcal{H}(|x|-ct)$ for a positive constant speed $c>0$, for a given positive release function $\mathcal{H}$. 
It has been proved in \cite{ALM3} that for a simple scalar reaction-diffusion equation and in one dimension, there exists a speed $c>0$ and a profile $\mathcal{H}$ such that the population goes to extinction. 
In the monostable situation and still in one dimension, a similar study has been performed in \cite{LEC2023} for system \eqref{syst}. A first difficulty to extend this result to the bistable situation is to find a condition under which there is a natural propagation of mosquitoes without sterile males. Indeed it is clear that when $\Gamma\equiv 0$ there is no need to release sterilized males to eradicate the population. A second technical difficulty lies in the construction of the forced wave of eradication for the whole system. We will use a comparison principle and we will construct sub-solutions and super-solutions for this system.
We will assume the spherical symmetry to investigate the two-dimensional case.

% !!!!! A COMPLETER : Préciser la notion de solution : sup des sous-solution ? [Volpert?] !!!!!

The outline of the paper is the following. In the next section, we state our main results. The first results concern the model without sterile male for which we provide a condition on the parameter to guarantee the invasion of the species (Proposition \ref{prop:TWMs0}). Then we present in Theorem \ref{MainTheorem} the result concerning the existence of a wave of eradication of the species by acting on a moving frame. Section \ref{sec:Ms0} is devoted to the proof of Proposition \ref{prop:TWMs0}. The proof of Theorem \ref{MainTheorem} is divided into three parts :  in Section \ref{sec:sub-solution}, we construct a sub-solution; in Section \ref{sec:super-solution}, we construct a super-solution; finally, the proof is concluded using these super- and sub-solutions in Section \ref{sec:final}. In an Appendix, we propose an analysis of the steady states as stated in Lemma \ref{lem:stat}.

\section{Main results}

In this section, we state our main results.
Existence of a unique solution of system \ref{syst} may be obtained by using the classical theory of nonlinear parabolic systems and has already been obtained for such a system in \cite{ANG}.
The aim of this work is to prove the existence of 'forced' traveling waves to eradicate the population of insects in a two-dimensional domain.
Before presenting the existence of such 'forced' traveling waves, it is important to state some results for the system without sterile males.

\subsection{Presentation of the main results : Case without sterile males}

When there are no sterile males (i.e. $M_s=0$), the system simplifies into
\begin{subequations}\label{syst0}
\begin{align}\label{syst0:E}
&\partial_t E = b F \left(1 - \frac{E}{K}\right) - (\mu_E+\nu_E) E \\
\label{syst0:M}
&\partial_t M - \Delta M = (1-\rho) \nu_E E - \mu_M M  \\
\label{syst0:F}
&\partial_t F - \Delta F = \rho \nu_E E \Gamma(M) - \mu_F F. 
\end{align}
\end{subequations}
Let us introduce the basic reproduction number $\mathcal{N}$ and a parameter $\zeta$ defined by 
\begin{equation}\label{def:Nzeta}
    \mathcal{N} := \frac{b\rho\nu_E}{\mu_F (\nu_E+\mu_E)}, \quad 
\zeta := \frac{\mu_M }{(1-\rho) \nu_E \gamma K}.
\end{equation}
We will always assume that $\mathcal{N}>1$, meaning that the population does not go to extinction naturally.

The dynamical system without diffusion corresponding to \eqref{syst0} reads
\begin{subequations}\label{systODE}
\begin{align}\label{systODE:E}
&E' = b F \left(1 - \frac{E}{K}\right) - (\mu_E+\nu_E) E \\
\label{systODE:M}
&M' = (1-\rho) \nu_E E - \mu_M M  \\
\label{systODE:F}
&F' = \rho \nu_E E \Gamma(M) - \mu_F F. 
\end{align}
\end{subequations}
The following Lemma provides some properties of the equilibria of this ODE system and their stability and justifies the terminology 'monostable' for $\Gamma\equiv 1$ and 'bistable' for $\Gamma(M) = 1-e^{-\gamma M}$ (see \eqref{eq:Gamma}).
\begin{lemma}\label{lem:stat}
Assume $\mathcal{N}>1$. 
\begin{itemize}
\item[(i)] In the monostable case $\Gamma(M)=1$, system \eqref{systODE} has two steady states : The extinction equilibrium $(0,0,0)$ which is unstable and a positive equilibrium $(E^*,M^*,F^*)$ which is stable and is given by
$$
F^* = \frac{K(\mu_E+\nu_E)}{b}(\mathcal{N}-1), \quad E^* = \frac{\mu_F}{\rho \nu_E} F^*,\quad 
M^* = \frac{(1-\rho) \mu_F}{\rho \mu_M} F^*.
$$
\item[(ii)] In the bistable case $\Gamma(M) = 1-e^{-\gamma M}$,
let $\zeta_c$ be the unique positive solution of the equation 
\begin{equation}\label{def:zetac}
 \frac{1+\sqrt{4\zeta_c\mathcal{N}+1}}{2\mathcal{N}} = 1-\zeta_c \ln \left(\frac{2\zeta_c \mathcal{N} + 1 + \sqrt{4\zeta_c \mathcal{N} + 1}}{2\zeta_c \mathcal{N}} \right).
\end{equation} 
If $\zeta<\zeta_c$, or equivalently $\gamma>\gamma_c:=\dfrac{\mu_M}{(1-\rho)\nu_E \zeta_c K}$, then the system \eqref{systODE} admits three constant stationary solutions~: The extinction equilibrium $(0,0,0)$ and two positive equilibria $(E^*_1,M^*_1,F^*_1) \ll (E^*,M^*,F^*)$.
Moreover the extinction equilibrium and the positive equilibrium $(E^*,M^*,F^*)$ are locally asymptotically stable, whereas the equilibrium $(E^*_1,M^*_1,F^*_1)$ is unstable.
\end{itemize}
\end{lemma}
The proof of the first point $(i)$ may be obtained by straightforward computations. The proof of the second point $(ii)$ is postponed to the appendix; we also refer to \cite{STR19}. The main idea of the proof is the following remark: if there exists a stationary solution $(E^*, F^*, M^*)$  then direct computations imply necessarily that
\begin{equation}\label{eq:gamma:phi}
\Gamma(\phi_0(F^*)) = \frac{\mu_F F^*}{\rho \nu_E K} + \frac{1}{\mathcal{N}}
\end{equation}
with 
\begin{equation}\label{def:phi0}
\phi_0(F) = \frac{(1-\rho) \nu_E b F}{\mu_M\frac{b F}{K}+\mu_M(\mu_E+\nu_E)}.
\end{equation}
The idea is then to prove that if $\gamma > \gamma_c$ then \eqref{eq:gamma:phi} has two positive solutions.  

\medskip

% \begin{equation}
%   \label{cond:stat}
%  \frac{1+\sqrt{4\zeta\mathcal{N}+1}}{2\mathcal{N}} < 1-\zeta \ln \left(\frac{2\zeta \mathcal{N} + 1 + \sqrt{4\zeta \mathcal{N} + 1}}{2\zeta \mathcal{N}} \right).
% \end{equation}

We focus now on the invasion.
It is well-known that in the monostable case, there is a 'hair trigger effect' meaning that as soon as the initial data is non-zero and nonnegative, then the species is invading, i.e. the solution of \eqref{syst0} converges to the positive steady state in the whole domain. 
The bistable case is more sophisticated. Indeed, since both the extinction equilibrium and the positive equilibrium $(E^*,M^*,F^*)$ are stable, it is not clear to which of these steady states the solution will converge on the whole spatial domain. However, there exists planar traveling wave solutions connecting the two stable steady states $(0,0,0)$ and $(E^*,M^*,F^*)$ (see e.g. \cite{VOL}), i.e. a direction $\textrm{e}$ and a particular solution under the form $(E,M,F)(x,t) = (\widetilde{E},\widetilde{M},\widetilde{F})(x\cdot\textrm{e}+ct)$ where $c$ is the so-called speed of the front, and $\widetilde{E},\widetilde{M},\widetilde{F}$ are nondecreasing functions from $\RR$ to $\RR^+$ such that
\begin{subequations}\label{systTW}
\begin{align}\label{systTW:E}
&c \widetilde{E}' = b \widetilde{F} \left(1 - \frac{\widetilde{E}}{K}\right) - (\mu_E+\nu_E) \widetilde{E} \\
\label{systTW:M}
&c \widetilde{M}' - \widetilde{M}'' = (1-\rho) \nu_E \widetilde{E} - \mu_M \widetilde{M}  \\
\label{systTW:F}
  &c \widetilde{F}' - \widetilde{F}'' = \rho \nu_E \widetilde{E}\, \Gamma(\widetilde{M}) - \mu_F \widetilde{F}  \\
  \label{systTW:lim}
  & (\widetilde{E},\widetilde{M},\widetilde{F})(+\infty) = (0,0,0), \qquad
     (\widetilde{E},\widetilde{M},\widetilde{F})(-\infty) = (E^*,M^*,F^*).
\end{align}
\end{subequations}
With this convention, we say that the species is invasive when $c>0$.

Although determining the sign of the speed $c$ for scalar reaction-diffusion is well-known since decades (see e.g. \cite{PER}), the case of systems is more tricky and still widely open (we refer e.g. to the review article \cite{Girardin}).
The following proposition gathers our main results concerning traveling waves in the bistable case and states a sufficient condition on the parameters to guarantee invasion of the species.
\begin{proposition}\label{prop:TWMs0}
  Let us assume that $\mathcal{N}>1$ and consider the bistable case $\Gamma(M) = (1-e^{-\gamma M})$ with $\gamma>\gamma_c$. Then, there exists a traveling wave solution $(c,\widetilde{E},\widetilde{M},\widetilde{F})$ of \eqref{systTW}.
  Moreover, $\gamma\mapsto c_\gamma$ is increasing and there exists $\gamma_0>\gamma_c$ such that,
  for $\gamma>\gamma_0$ it holds that $c_\gamma>0$, and $\gamma_0$ is such that
  \begin{equation}
      \label{eq:condition}
  \int_0^{F^*} \left(\frac{\rho \nu_E b u}{\frac{bu}{K}+\mu_E+\nu_E} \left(1-e^{-\gamma_0\phi(u)}\right) - \mu_F u\right)\,du = 0,
  \end{equation}
  where the function $\phi$ is defined in \eqref{def:phi}.
\end{proposition}
This Proposition will be proved in Section \ref{sec:Ms0}. The fact that relation \eqref{eq:condition} defines $\gamma_0$ uniquely is also addressed in Remark \ref{rem:condition}. 
We point out that the condition $\gamma >\gamma_0$ is a sufficient condition to guarantee the invasion phenomena, not a necessary one.
% \textcolor{blue}{\textbf{Je pense même que la bonne condition est 
% $$
%  \int_0^{F^*} \left(\frac{\rho \nu_E b u}{\frac{bu}{K}+\mu_E+\nu_E} \left(1-e^{-\gamma_0\phi_0(u)}\right) - \mu_F u\right)\,du = 0.
%  $$
%  En effet, le choix de la division par 2 est arbitraire est dû au lemme 5. On peut l'améliorer en ayant $(1+\varepsilon) $ dans l'exposant de l'exponentielle pour tout $\varepsilon >0$. Il faudrait faire du numérique et voir si cela vaut la peine d'être explicité ou peut être juste donner l'idée de démonstration étant donné que ce n'est pas le résultat principal et que le papier est déjà lourd comme cela. }}
%  \textcolor{red}{C'est certainement vrai, mais je ne sais pas vraiment si cela vaut le coup d'écrire ce résultat dans cet article, sachant que dans tous les cas, cela ne sera qu'une condition suffisante ! Le papier est assez dense comme ça.}

\subsection{Presentation of the main results : General case}

As a consequence of Proposition \ref{prop:TWMs0}, we have that for $\gamma>\gamma_0$ the mosquito species modeled by system \eqref{syst} is invasive. Then, it is relevant to use the sterile insect technique to fight against this invasion.
We consider now the full system \eqref{syst} where we assume that sterile males are released in an annulus of action of width $L=r_2-r_1$.

We first recall an important and useful result concerning the equilibria and their stability for the corresponding dynamical system. This system reads~:
\begin{subequations}\label{systdyn}
\begin{align}\label{systdyn:E}
&\frac{dE}{dt} = b F \left(1 - \frac{E}{K}\right) - (\mu_E+\nu_E) E \\
\label{systdyn:M}
&\frac{dM}{dt} = (1-r) \nu_E E - \mu_M M  \\
\label{systdyn:F}
&\frac{dF}{dt} = r \nu_E E \frac{M}{M+\gamma_s M_s}\Gamma(M+\gamma_s M_s) - \mu_F F  \\
\label{systdyn:Ms}
&\frac{dM_s}{dt} = \Lambda  - \mu_s M_s 
\end{align}
\end{subequations}
The following result shows that the sterile insect technique may be efficient to eradicate the population of mosquitoes~: 
\begin{lemma}\label{lem:0GAS}
  Under the assumption of Lemma \ref{lem:stat}. Let us consider the differential system \eqref{systdyn} with the function $\Gamma$ as in \eqref{eq:Gamma}. There exists $\Lambda^*>0$ such that if $\Lambda > \Lambda^*$ then the extinction equilibrium $(0,0,0)$ is globally attractive for \eqref{syst:E}--\eqref{syst:F}.
\end{lemma}
We refer to \cite[Lemma 3]{STR19} for the bistable case and \cite[Proposition 2.1]{ADPV2022} for the proof of this result.

In order to obtain the existence of 'forced' traveling waves, we must assume that the initial data are 'well-prepared'. More precisely, we assume that mosquitoes have been eliminated in the center of the domain (for instance by applying the SIT in a fixed region) whereas the mosquito population is at the positive equilibrium far away from the center. 
Such an assumption is natural. Indeed, if we want to propagate an elimination strategy to a large area, we need to be sure of having succeeded in eliminating in a smaller region before extending the strategy.
More precisely, we will assume the following~:
\begin{subequations}\label{hyp:init}
\begin{align}
    \label{hyp:init1}
    \exists\, R_0^0 >0, C_0 >0, u_0\in (0,1), \quad \forall\,x\in \RR^2, \quad 0 \leq F^0(x) \leq  F^* (u_0 \mathbf{1}_{\{|x|\leq R_0^0\}} + \mathbf{1}_{\{|x|>R_0^0\}}), \\[2mm]
    0\leq E^0(x) \leq \min\{K,C_0 F^0(x)\}, 
    \quad M^0(x) \leq C_0 F^0(x), \nonumber \\ 
    \quad M_s^0 \geq \frac{\bar\Lambda}{\mu_s} \mathbf{1}_{\{|x|\leq R_0^0\}} \text{ and } M_s^0\in L^\infty(\RR^2), \nonumber   \\
    \label{hyp:init2}
    \exists\, R_0^1 > 0, \quad \forall\, |x| > R_0^1, \qquad (E^0,M^0,F^0,M_s^0)(x) = (E^*,M^*,F^*,0).
\end{align}
\end{subequations}
where $(E^*,M^*,F^*)$ is the largest equilibrium defined in Lemma \ref{lem:stat}, and $\bar\Lambda$ is as in the statement of Theorem \ref{MainTheorem} below.
The main result of this work concerns the existence of a wave of extinction. 
\begin{theorem}\label{MainTheorem}
Let us assume $\mathcal{N}>1$, and $\gamma>\gamma_0$ in the bistable case, where $\gamma_0$ is defined in Proposition \ref{prop:TWMs0}.
  Let $c>0$, $0<R_1<R_2$. Let us assume that the release function is given by
  \begin{subequations}\label{Lambda}
    \begin{align}\label{Lambdabistable}
      &\Lambda(x,t) = \bar\Lambda \mathbf{1}_{\{R_1+ct\leq |x|\leq R_2+ct\}}, \quad \text{ in the bistable case},  \\
      \label{Lambdamonostable}
      &\Lambda(x,t) = \bar\Lambda \mathbf{1}_{\{R_1+ct\leq |x|\leq R_2+ct\}} + \bar\Lambda e^{\eta(|x|-(R_1+ct))} \mathbf{1}_{\{|x| < R_1+ct\}} , \quad \text{ in the monostable case}.
    \end{align}
  \end{subequations}
  
  Then, there exist $\bar\Lambda>0$ large enough, $\bar\eta$ small enough and  $\bar L>0$ large enough such that for all $\Lambda \geq \bar\Lambda$, $R_2-R_1 > \bar L$, $0<\eta\leq \bar \eta$, the solution of \eqref{syst} with the release function given in \eqref{Lambda} and initial data satisfying \eqref{hyp:init} for $R_0^0<R_0^1$ large enough and $u_0$ small enough, verifies
\begin{equation*}
\begin{aligned}
&(i) \quad \forall\, \underline{c} < c, \quad \underset{ t \to +\infty}{\lim} \ \underset{|x| < \underline{c} t}{\sup} \|(E,M,F)(x, t)\| = 0, \\
&(ii) \quad  \forall\, \overline{c} > c, \quad \underset{ t \to +\infty}{\lim} \ \underset{|x| > \overline{c}t}{\inf} \|(E^*,M^*,F^*)-(E,M,F)(x, t)\| = 0.
\end{aligned}
\end{equation*}
\end{theorem}

For practical applications, it seems natural to consider the heterogeneous case where the carrying capacity $K$ depends on the space variable. More precisely, let us assume the following~:
\begin{equation}\label{hyp:Khetero}
\exists\, K_2>K_1>0, \text{ such that, for all } x \in \RR^2,\  K_1 \leq K(x) \leq K_2.
\end{equation}
Then, we keep the same assumption on the initial data as in \eqref{hyp:init}, except that we modify obviously the assumption on $E^0$ in the following way~:
\begin{equation}\label{hyp:varKE}
0 \leq E^0(x) \leq \min\{K(x),C_0 F^0(x)\}\ \text{ for all } x\in \RR^2.
\end{equation}
As a consequence of Theorem \ref{MainTheorem} we have
\begin{corollary}\label{cor:hetero}
    Under the assumptions of Theorem \ref{MainTheorem} taking into account the modification \ref{hyp:varKE}, let $c>0$ and $0<R_1<R_2$ and consider the release function $\Lambda$ as in \eqref{Lambda}. Then, there exist $\bar\Lambda$ large enough, $\bar\eta$ small enough and  $\bar L>0$ large enough such that for all $\Lambda \geq \bar\Lambda$, $R_2-R_1 > \bar L$, $0<\eta\leq \bar \eta$, the solution of \eqref{syst} with the release function given in \eqref{Lambda} and initial data satisfying \eqref{hyp:init} for $R>R_2$ and $u_0$ small enough, verifies
\begin{equation*}
\begin{aligned}
&(i) \quad \forall \underline{c} < c, \quad \underset{ t \to +\infty}{\lim} \ \underset{|x| < \underline{c} t}{\sup} \|(E,M,F)(x, t)\| = 0, \\
&(ii) \quad  \forall \overline{c} > c, \quad \underset{ t \to +\infty}{\lim} \ \underset{|x| > \overline{c}t}{\inf} (E,M,F)(x, t) > 0.
\end{aligned}
\end{equation*}
\end{corollary}

Theorem~\ref{MainTheorem} has potential applications in real-world field implementations. One of the key limitations of the Sterile Insect Technique (SIT) is the daily production capacity of sterile males. Compared to a "naive strategy" in which health authorities release sterile males over a growing disc of radius $R + ct$ (i.e., the region $\lbrace |x| \leq R + ct \rbrace$), the "annulus strategy" allows for coverage of a larger area using the same or fewer resources. Indeed, over a fixed time interval $[0, T]$, the "naive strategy" requires $O(T^3)$ sterile males, whereas the "rolling carpet strategy" only requires $O(T^2)$. In fact, for each strategy, the number of released sterile males during an interval $[0,T]$ denoted by $\mathcal{M}_s$ is gven by:
\begin{itemize}
    \item \textit{For the naive strategy :}
    \[ \mathcal{M}_s = \int_0^T \overline{\Lambda} \pi (r+ct)^2dt = O(T^3) \]
     \item \textit{For the annulus strategy : (this computation corresponds to the bistable and is similar for the monostable case)}
    \[ \mathcal{M}_s = \int_0^T m_s \pi\left( (r_1+ct)^2 - (r_2 + ct)^2\right)dt = \int_0^T \overline{\Lambda}  \pi (r_1-r_2)(r_1+r_2+2ct)dt = O(T^2).\]    
\end{itemize}

\subsection{Idea of the proof}

We first observe that due to the monotony of the system, there is a comparison principle for system \ref{syst} on the invariant set $[0,K]\times \RR^3_+$~:
\begin{lemma}\label{lem:invariant}
    The set $[0,K]\times \RR^3_+$ is invariant, i.e. if $0\leq E^0 \leq K$, $0\leq F^0$, $0\leq M^0$, $0\leq M_s^0$ then for all $t>0$, the solution of \eqref{syst} verifies $0\leq E(t) \leq K$, $0\leq F(t)$, $0\leq M(t)$, $0\leq M_s^0$.
\end{lemma}
Notice that since the equation on $E$ does not have partial derivatives in the $x$ variable, the result of Lemma \ref{lem:invariant} is also true when $K$ is a function of $x$ and verifies \eqref{hyp:Khetero}.

Denoting,
\begin{align*}
&f_E(E,F,M,M_s) = b F \left(1 - \frac{E}{K}\right) - (\mu_E+\nu_E) E,
\qquad f_M(E,F,M,M_s) = (1-r) \nu_E E - \mu_M M,  \\
&f_F(E,F,M,M_s) = r \nu_E E \frac{M}{M+\gamma_s M_s}\Gamma(M+\gamma_s M_s) - \mu_F F, \qquad    
f_s(E,F,M,M_s) = \Lambda  - \mu_s M_s.
\end{align*}
We may rewrite system \eqref{syst} in the compact form
$$
\partial_t \mathbf{U} - \mathbb{D}\Delta \mathbf{U} = \mathbf{f}(\mathbf{U}):=
\begin{pmatrix}
    f_E(\mathbf{U}) \\ f_M(\mathbf{U})  \\  f_F(\mathbf{U})  \\ f_s(\mathbf{U})
\end{pmatrix},\ \text{ with } \mathbf{U} = \begin{pmatrix}
    E \\ M  \\ F \\ M_s
\end{pmatrix}, \quad \mathbb{D} = \begin{pmatrix}
    0 & 0 & 0 & 0 \\ 0 & D & 0 & 0  \\  0 & 0 & D & 0  \\  0 & 0 & 0 & D
\end{pmatrix}.
$$
After straightforward computations, we get
\begin{align*}
    &\frac{\partial f_E}{\partial F} \geq 0, \quad \frac{\partial f_M}{\partial E} \geq 0, \quad
    \frac{\partial f_F}{\partial E} \geq 0, \\
    &\frac{\partial f_F}{\partial M} = r\nu_E E \left(\frac{\gamma_s M_s}{(M+\gamma_s M_s)^2} \Gamma(M+\gamma_s M_s) + \frac{\gamma_s M_s}{M+\gamma_s M_s}\Gamma'(M+\gamma_s M_s)\right) \geq 0, \\
    &\frac{\partial f_F}{\partial M_s} = -r\nu_E E \frac{\gamma_s M}{(M+\gamma_s M_s)^2}\left(\Gamma(M+\gamma_s M_s) -(M+\gamma_s M_s)\Gamma'(M+\gamma_s M_s)\right).
\end{align*}
Clearly, with the choice of $\Gamma$ in the monostable case \eqref{eq:Gamma}, we have $\frac{\partial f_F}{\partial M_s}\leq 0$. In the bistable case, we compute
$$
\frac{\partial f_F}{\partial M_s} = -r\nu_E E \frac{\gamma_s M}{(M+\gamma_s M_s)^2}\left(1-e^{-\gamma(M+\gamma_s M_s)}(1+\gamma(M+\gamma_s M_s)\right) \leq 0,
$$
where we use the well-know inequality $1+x \leq e^x$. 

A consequence of these computations is that the system is monotone for the order relation of the cone $\RR^3_+\times \RR_-$~:
\begin{definition}
\begin{itemize}
    \item[(i)] For any vector $\mathbf{u}, \mathbf{v} \in \RR^4$, we define a partial order $\preceq$ such that $\mathbf{u}\preceq \mathbf{v}$ if and only if $u_i \leq v_i$ for $i\in \{1,2,3\}$ and $u_4\geq v_4$.
    \item[(ii)] We say that $\overline{\mathbf{U}}=(\barE,\barF,\barM,\barM_s)$ is a \textbf{super-solution} of system \eqref{syst}, if it verifies, in the distributional sense, $\partial_t \overline{\mathbf{U}} - \mathbb{D} \Delta \overline{\mathbf{U}} \succeq \mathbf{f}(\overline{\mathbf{U}})$ and $\overline{\mathbf{U}}(t=0) \succeq (E^0,F^0,M^0,M_s^0)$.

    We say that $\underline{\mathbf{U}}=(\underline{E},\underline{F},\underline{M},\underline{M_s})$ is a \textbf{sub-solution} of system \eqref{syst}, if it verifies, in the distributional sense, $\partial_t \underline{\mathbf{U}} - \mathbb{D} \Delta \mathbf{\underline{U}} \preceq \mathbf{f}(\underline{\mathbf{U}})$ and $\underline{\mathbf{U}}(t=0) \preceq (E^0,F^0,M^0,M_s^0)$.    
\end{itemize}
\end{definition}
It is standard to deduce the following comparison principle see e.g. \cite{PER,LEC2023}.
\begin{lemma}[Comparison principle]\label{lem:compar}
    Let us consider 
    $$
    0\leq E_1^0 \leq E_2^0 \leq K, \quad 0\leq M_1^0 \leq M_2^0, \quad 0\leq F_1^0 \leq F_2^0, \quad 0\leq M_{s,2}^0 \leq M_{s,1}^0.
    $$
    Suppose that $U_1 := (E_1,M_1,F_1,M_{s,1})$ is a sub-solution of \eqref{syst} with initial data $U_1^0 := (E_1^0,M_1^0,F_1^0,M_{s,1}^0)$, and $U_2 := (E_2,M_2,F_2,M_{s,2})$ is a super-solution of \eqref{syst} with initial data $U_2^0 := (E_2^0,M_2^0,F_2^0,M_{s,2}^0)$.
    Then, for all $t>0$, we have $U_1 \preceq U_2$.
\end{lemma}

The idea of the proof of Theorem \ref{MainTheorem} is to use the classical sub- and super-solution technique. 

More precisely, for the sub-solution, we first construct an invading sub-solution in the case without sterile males (i.e. $m_s = 0$). In this order, we put the equation of the eggs at equilibrium and manage to find a sub-solution $(\underline{M}, \underline{F})$ mainly driven by $\underline{F}$. This sub-solution allows to prove Proposition \ref{prop:TWMs0}. Then, following similar arguments, we extend this kind of argument for the case where the sterile population is small $0<m_s\ll 1$, i.e. in the region where $|x|+ct$ large enough. 

Next, for the super-solution, we look for a radially symmetric super-solution that goes to $0$ in the set $\lbrace |x| < \underline{c}t \rbrace$ for any $\underline{c}<c$. 
To do so, we split the spatial domain into four subdomains. Let $\underline{c}<c'<c$ and $0<R_1<r_1<r_2<R_2$.
\begin{enumerate}
\item $\Omega^0_t = B_{r_1 + c't}$ (where $B_r$ denotes the ball of radius $r$ and center $0$) with $c' \in (\frac{\underline{c} + c}{2}, c)$ that will be fixed later on,
\item $\Omega^1_t = T(0, r_1 + c't, r_1 + ct )$  (where $T(z,r,R)$ denotes the annulus of center $z$, small radius $r$ and big radius $R$, i.e. $T(z,r,R)=\{x\in\RR^2,\ r\leq \|x-z\| \leq R\}$),
\item $\Omega^2_t  = T(0, r_1 + ct, r_2 + ct)$ (it is the annulus of action), 
\item $\Omega^3_t = B_{r_2 + ct }^c$, the rest of the field. 
\end{enumerate}
Notice that $\mathbb{R}^2 = \overline{\Omega^0_t \cup \Omega^1_t\cup \Omega^2_t \cup \Omega^3_t}$. We underline that the distance $L = r_2 - r_1$ is not fixed yet. 

As mentioned above, since we suppose the diffusion to be constant, up to a rescaling, we may assume that the diffusion coefficient $D=1$. Therefore, for the sake of simplicity of the computations and the notations, we will always consider that $D=1$.

\section{Analysis of the model without sterile males}\label{sec:Ms0}

The aim of this section is to prove Proposition \ref{prop:TWMs0}.

\subsection{Stationary solution in a half space}
Let us consider the existence of stationary solutions in one dimension on $(0,+\infty)$. More precisely, we study the following system on $(0,+\infty)$
\begin{subequations}\label{steady}
\begin{align}\label{steady:E}
&0 = b F \left(1 - \frac{E}{K}\right) - (\mu_E+\nu_E) E \\
\label{steady:M}
& - M'' = (1-\rho) \nu_E E - \mu_M M  \\
\label{steady:F}
& - F'' = \rho \nu_E E \Gamma(M) - \mu_F F,
\end{align}
\end{subequations}
complemented with initial conditions $(E(0),M(0),F(0))=(0,0,0)$.
Notice that this system reduces to
$\displaystyle E = \frac{bF}{\frac{bF}{K}+\mu_E+\nu_E}$ and
\begin{subequations}\label{stat}
\begin{align}\label{stat:M}
  & - M'' = \frac{(1-\rho) \nu_E b F}{\frac{bF}{K}+\mu_E+\nu_E} - \mu_M M  \\
  \label{stat:F}
  & - F'' = \frac{\rho \nu_E b F}{\frac{bF}{K}+\mu_E+\nu_E} \Gamma(M) - \mu_F F.
\end{align}
\end{subequations}
We want to prove that under certain conditions on $\Gamma$, there exists a solution of \eqref{stat} in $(0,+\infty)$ which is such that $(M(0),F(0))=(0,0)$, $(M(+\infty),F(+\infty)) = (M^*,F^*)$, and $M$ and $F$ are nondecreasing on $(0,+\infty)$.
\begin{lemma}\label{lem:1}
  Let $\mu>0$ and $\psi$ be a nondecreasing continuous function on $(0,+\infty)$ with $\lim_{x\to +\infty} \psi(x) = \psi_\infty$. Then, there exists a nondecreasing solution of
  $$
  -u'' + \mu u =  \psi(x), \qquad u(0) = 0, \quad u(+\infty) = \frac{\psi_\infty}{\mu}.
  $$
  Moreover, we have the estimate
  $$
  \forall\, x\in(0,+\infty),\qquad u(x) \geq \frac{1}{2\mu} \psi(x) \left(1-e^{-2\sqrt{\mu}x}\right).
  $$
\end{lemma}
\begin{proof}
  Indeed, after straightforward computations, the solution is given by the expression
  $$
  u(x) = \frac{1}{\sqrt{\mu}} \left(\int_x^{+\infty} \psi(y) e^{-\sqrt{\mu}y}\,dy \sinh(\sqrt{\mu} x) + \int_0^x \psi(y) \sinh(\sqrt{\mu}y)\,dy e^{-\sqrt{\mu}x}\right).
  $$
  From this expression, we clearly deduce from the nonnegativity of $\psi$ that $u(x)\geq 0$ for any $x\in(0,+\infty)$.
  Then, since $\psi$ is nondecreasing we have
  $$
  u(x) \geq \frac{1}{\sqrt{\mu}} \int_x^{+\infty} \psi(x) e^{-\sqrt{\mu}y}\,dy \sinh(\sqrt{\mu} x)
  = \frac{\psi(x)}{\mu} e^{-\sqrt{\mu}x} \sinh(\sqrt{\mu}x).
  $$
  This is the desired estimate.
  Finally, computing the derivative we obtain
  $$
  u'(x) = \int_x^{+\infty} \psi(y) e^{-\sqrt{\mu}y}\,dy \cosh(\sqrt{\mu} x) - \int_0^x \psi(y) \sinh(\sqrt{\mu}y)\,dy e^{-\sqrt{\mu}x}.
  $$
  Using again the fact that $\psi$ is nondecreasing, we get
  $$
  u'(x) \geq \psi(x) \left(\int_x^{+\infty} e^{-\sqrt{\mu}y}\,dy \cosh(\sqrt{\mu} x) - \int_0^x \sinh(\sqrt{\mu}y)\,dy e^{-\sqrt{\mu}x}\right) = \frac{1}{\sqrt{\mu}} e^{-\sqrt{\mu}x} > 0.
  $$
  Hence $u$ is nondecreasing.
\end{proof}

\begin{lemma}\label{lem:2}
  Assume $(M,F)$ is a solution of \eqref{stat} such that $0\leq M \leq M^*$ and $0\leq F \leq F^*$ and $M(0)=0$, $F(0)=0$. Then, we have
  $$
  M \leq M^* (1-e^{-\sqrt{\mu_M} x}), \qquad
  F \leq F^* (1-e^{-\sqrt{\mu_F} x}).
  $$
\end{lemma}
\begin{proof}
  Indeed, under the assumptions of the Lemma, we have
  $$
  \frac{\rho \nu_E b F}{\frac{bF}{K}+\mu_E+\nu_E} \Gamma(M) \leq \frac{\rho \nu_E b F^*}{\frac{bF^*}{K}+\mu_E+\nu_E} \Gamma(M^*) = \mu_F F^*.
  $$
  Therefore, the solution of the equation
  $$
  -\bar{F}'' = \mu_F (F^* - \bar{F}), \qquad  \overline{F}(0) = 0, \quad \overline{F}(+\infty) = F^*
  $$
  is a super-solution of \eqref{stat:F}. Hence,
  $$
  F(x) \leq \bar{F}(x) = F^* (1-e^{-\sqrt{\mu_F}x}).
  $$
  The proof is the same for the estimate on $M$.
\end{proof}

Using these two preliminary results we obtain an interesting estimate: if $F$ is nondecreasing on $(0,+\infty)$, we deduce that the function
$$
\psi(x) = \frac{(1-\rho) \nu_E b F(x)}{\frac{bF(x)}{K}+\mu_E+\nu_E} 
$$
is nondecreasing and we may apply the result of Lemma \ref{lem:1}.
We deduce that
\begin{equation}\label{eq:phi:int1}
 \frac{1}{2\mu_M} \psi(x) \left(1-e^{-2\sqrt{\mu_M}x}\right) \leq M(x).
\end{equation}
Moreover, from Lemma \ref{lem:2}, we have
$$
F(x) \leq F^* \left(1-e^{-\sqrt{\mu_F}x}\right),
$$
which is equivalent to
\begin{equation}\label{eq:phi:int2}
x \geq - \frac{1}{\sqrt{\mu_F}} \ln \left(1-\frac{F}{F^*}\right).
\end{equation}
Finally, denoting
\begin{equation}\label{def:phi}
\phi(F) = \frac{1}{2\mu_M}\frac{(1-\rho) \nu_E b F}{\frac{bF}{K}+\mu_E+\nu_E} \left(1 - \exp\left(2 \sqrt{\frac{\mu_M}{\mu_F}}\ln(1-\frac{F}{F^*})\right)\right),
\end{equation}
and combining \eqref{eq:phi:int1} and \eqref{eq:phi:int2}, it follows that
$$
\phi(F) \leq M.
$$
Therefore, let us consider the following system in $(0,+\infty)$
\begin{subequations}\label{syst:subsol}
\begin{align}  \label{subsol:E}
  & \underline{E} = \frac{b \underline{F}}{\frac{b\underline{F}}{K}+\mu_E+\nu_E}   \\
  \label{subsol:M}
  & -\underline{M}'' = (1-\rho)\nu_E \underline{E} - \mu_M \underline{M}  \\
  \label{subsol:F}
& -\underline{F}'' = \frac{\rho \nu_E b \underline{F}}{\frac{b \underline{F}}{K} + \mu_E + \nu_E} \Gamma(\phi(\underline{F})) - \mu_F \underline{F},
\end{align}
\end{subequations}
complemented with initial condition $(\underline{E}(0),\underline{M}(0),\underline{F}(0))=(0,0,0)$. We recall that the expressions of $\phi$ and $\Gamma$ in the bistable case are given in \eqref{def:phi} and \eqref{eq:Gamma}.

\begin{proposition}\label{prop:exist}
  Let us assume $\gamma>\gamma_c$ and that the following condition holds
  \begin{equation}\label{condition}
    \int_0^{F^*} \left(\frac{\rho\nu_E b u}{\frac{b u}{K} + \mu_E + \nu_E} \Gamma(\phi(u)) - \mu_F u \right)\,du > 0.
  \end{equation}
  Then, system \eqref{syst:subsol} with zero initial condition admits a solution $(\underline{E},\underline{M}, \underline{F})$ which is bounded and nondecreasing (in the sense that $\underline{E}$, $\underline{M}$ and $\underline{F}$ are bounded and nondecreasing). Moreover, there exists $(\underline{E_m},\underline{M_m},\underline{F_m})>(E_1^*,M_1^*,F_1^*)$ such that $(\underline{E},\underline{M}, \underline{F})$ converges to $(\underline{E_m},\underline{M_m},\underline{F_m})$ at $x\to +\infty$.
\end{proposition}

\begin{proof}
  From Lemma \ref{lem:stat}, the condition $\gamma>\gamma_c$ guarantees the existence of $0<F_1^*<F^*$, stationary solutions.
  We first notice that it suffices to prove the result for the solution $\underline{F}$.
  Indeed, if there exists a solution $\underline{F}$ of \eqref{subsol:F} with $\underline{F}(0)=0$, $\underline{F}$ bounded and nondecreasing. 
  It is clear from \eqref{subsol:E} that $\underline{E}$ is bounded and nondecreasing.
  Using Lemma \ref{lem:1}, we deduce that there exists a bounded solution $\underline{M}$ of \eqref{subsol:M} with $\underline{M}(0)=0$ which is nondecreasing.

  Thus, let us consider equation \eqref{subsol:F}. We define
  $$
  G(F) = \int_0^{F} \left(\frac{\rho\nu_E b u}{\frac{b u}{K} + \mu_E + \nu_E} \Gamma(\phi(u)) - \mu_F u \right)\,du.
  $$
  The function $G$ is continuous on $[0,F^*]$ with $G(0) = 0$ and $G(F^*)>0$ by assumption \eqref{condition}. Thus there exists $\underline{F_m} \in (0,F^*]$ such that $G(\underline{F_m}) = \max_{[0,F^*]} G$ (if such a point is not unique, we define $\underline{F_m}$ as the smallest one such that for all $F\leq \underline{F_m}$, $G(F)<G(\underline{F_m})$).
  Then, we define $\underline{F}$ as the solution of the Cauchy problem
  $$
  \underline{F}' = \sqrt{2\Big(G(\underline{F_m})-G(\underline{F})\Big)}, \qquad \underline{F}(0)=0.
  $$
  By the Cauchy-Lipschitz theorem, there exists a unique solution to this equation, it is nondecreasing and bounded by $\underline{F_m}$ (indeed $F_m$ is a stationary solution). Then it admits a limit as $x$ goes to $+\infty$, which should be a stationary solution; the unique possible limit is $\underline{F_m}$.
  Moreover, this solution verifies
  $$
  \frac 12 (\underline{F}')^2 + G(\underline{F}) = G(\underline{F_m}).
  $$
  Deriving this expression, we get that $\underline{F}$ is a solution of \eqref{subsol:F}. 

  Then, we define $\underline{E_m}=\frac{b\underline{F_m}}{\frac{b\underline{F_m}}{K}+\mu_E+\nu_E}$ and $\underline{M_m} = \frac{(1-\rho)\nu_E}{\mu_M} \underline{E_m}$. 
  We construct $\underline{E}$ by \eqref{subsol:E} and $\underline{M}$ by solving \eqref{subsol:M} as stated in Lemma \ref{lem:1} such that $(\underline{E},\underline{M},\underline{F})$ converges to $(\underline{E_m},\underline{M_m},\underline{F_m})$ as $x$ goes to $+\infty$. To conclude the proof we are left to show the inequality $\underline{F_m}> F_1^*$, which implies straightforwardly the inequalities $\underline{E_m}> E_1^*$ and $\underline{M_m}> M_1^*$.
  We first observe from the definitions \eqref{def:phi0} and \eqref{def:phi} that, for all $F\in (0,F^*)$,
  $$
  \phi(F) \leq \phi_0(F).
  $$
  Hence, we deduce from \eqref{rel_stat3} (see Appendix) that in $(0,F_1^*)$ we have
  $$
  \frac{\rho\nu_E b}{\frac{b F}{K} + \mu_E + \nu_E} \Gamma(\phi_0(F)) - \mu_F < 0.
  $$
  From the monotony of $\Gamma$, we obtain then that on $(0,F_1^*)$ we have
  $$
  \frac{\rho\nu_E b}{\frac{b F}{K} + \mu_E + \nu_E} \Gamma(\phi(F)) - \mu_F < 0.
  $$
  Hence, we deduce from the expression of $G$ that $G(F)<0$ on $(0,F_1^*)$. 
  Since $G(\underline{F_m})>0$, we conclude that $F_1^*<\underline{F_m}$.
\end{proof}

\begin{remark}\label{rem:condition}
  Concerning condition \eqref{condition}, we first notice that actually it is enough to have the existence of a real $X>0$ such that
  $$
  \int_0^{X} \left(\frac{\rho\nu_E b u}{\frac{b u}{K} + \mu_E + \nu_E} \Gamma(\phi(u)) - \mu_F u \right)\,du > 0.
  $$
  Next, one may wonder whether it is possible to satisfy condition \eqref{condition}. Indeed, for instance, for $\gamma=0$ we have by definition $\Gamma=0$ and therefore \eqref{condition} can never be satisfied. However, we observe that when $\gamma\to +\infty$, the function $\Gamma(x)$ converges to $1$ for all $x>0$ and is bounded by $1$. Therefore, applying the dominated convergence theorem, we get that
  $$
  \lim_{\gamma\to +\infty} \int_0^{F^*} \frac{\rho\nu_E b u}{\frac{b u}{K} + \mu_E + \nu_E} \Gamma(\phi(u))\,du  = \int_0^{F^*} \frac{\rho\nu_E b u}{\frac{b u}{K} + \mu_E + \nu_E}\,du.
  $$
  And we verify easily that under the condition $\mathcal{N}>1$, \eqref{condition} is satisfied for $\Gamma=1$. As a consequence, \eqref{condition} holds true for $\gamma$ large enough.

  Moreover, since $\Gamma$ is increasing with respect to $\gamma$ and since it is proved in the Appendix that $F^*$ is also increasing with respect to $\gamma$, we notice that if \eqref{condition} holds for some $\gamma_0$, then it is also satisfied for any $\gamma>\gamma_0$ and that \eqref{eq:condition} allows to define $\gamma_0$ uniquely.
\end{remark}

\subsection{Proof of Proposition \ref{prop:TWMs0}}

We are now in position to construct a subsolution for system \eqref{steady}.
\begin{proposition}\label{prop:sub1D}
Let us assume $\gamma>\gamma_c$ and that \eqref{condition} holds.
Then, let us define for some $x_0\in \RR$ and for all $t>0$ and $x\in \RR$,
$$
\underline{\mathbf{E}}(t,x) = \underline{E}(x-x_0) \mathbf{1}_{x>x_0}; \quad
\underline{\mathbf{M}}(t,x) = \underline{M}(x-x_0) \mathbf{1}_{x>x_0}; \quad
\underline{\mathbf{F}}(t,x) = \underline{F}(x-x_0) \mathbf{1}_{x>x_0}.
$$
Then, for all $x_0\in \RR$, $(\underline{\mathbf{E}},\underline{\mathbf{M}},\underline{\mathbf{F}})$ is a subsolution of system \eqref{syst0} complemented with an initial data which is below this subsolution for some $x_0$.
\end{proposition}
\begin{proof}
    From Proposition \ref{prop:exist}, we know that $(\underline{\mathbf{E}},\underline{\mathbf{M}},\underline{\mathbf{F}})$ is well-defined and continuous on $(0,+\infty)\times \RR$ and is nondecreasing with respect to $x$.
    Then, we verify that it is a subsolution for each equation of this system:
    For the equation on $E$ it is obvious.
  For the equation on $M$, it is clear for $x<x_0$. For $x>x_0$, we compute 
  $$
  \partial_t \underline{\mathbf{M}} - \partial_{xx} \underline{\mathbf{M}} = - \underline{M}''(x-x_0) = (1-\rho)\nu_E \underline{E}(x-x_0) - \mu_M \underline{M}(x-x_0),
  $$
  where we use \eqref{subsol:M}. Thus, $\underline{\mathbf{M}}$ is a subsolution for $x>x_0$. Since $\partial_x \underline{\mathbf{M}}(x_0^+) = \underline{M}'(x_0^+)\geq 0 = \partial_x \underline{\mathbf{M}} (x_0^-)$ , it is also a subsolution on $\RR$.
  Moreover, applying Lemma \ref{lem:1}, we have for $x>x_0$
  \begin{equation}\label{ineqM1D}
    \underline{M}(x-x_0) \geq \frac{(1-\rho)\nu_E \underline{E}(x-x_0)}{2\mu_M}\left(1-e^{-2\sqrt{\mu_M} (x-x_0)}\right).
  \end{equation}
  
  Finally, for the equation on $F$, we compute, for $x>x_0$,
  \begin{equation}\label{eq:estimFsub1D}
    \partial_t \underline{\mathbf{F}} - \partial_{xx} \underline{\mathbf{F}}
    = -\underline{F}''(x-x_0) = \frac{\rho \nu_E b \underline{F}(x-x_0)}{\frac{b \underline{F}(x-x_0)}{K} + \mu_E + \nu_E} \Gamma\big(\phi(\underline{F}(x-x_0))\big) - \mu_F \underline{F}(x-x_0),
  \end{equation}
  where we use \eqref{subsol:F}.
  By definition of $\phi$ in \eqref{def:phi}, we have that $\phi$ is nondecreasing on $(0,F^*)$ as the product of two nonnegative nondecreasing functions. And
  $$
  \lim_{F\to F^*} \phi(F) = \frac{1}{2\mu_M} \frac{(1-\rho)\nu_E b F^*}{\frac{bF^*}{K} + \mu_E+\nu_E} = \frac{M^*}{2},
  $$
  where we use \eqref{rel_stat1} for the last equality.
  Then, recalling that $0\leq \underline{F} \leq \underline{F_m} \leq F^*$, we have since $\Gamma$ is nondecreasing
  $$
  \frac{\rho \nu_E b \underline{F}}{\frac{b \underline{F}}{K} + \mu_E + \nu_E} \Gamma(\phi(\underline{F})) \leq   \frac{\rho \nu_E b F^*}{\frac{b F^*}{K} + \mu_E + \nu_E} \Gamma(M^*/2) \leq 
  \frac{\rho \nu_E b F^*}{\frac{b F^*}{K} + \mu_E + \nu_E} \Gamma(M^*) = \mu_F F^*.
  $$
  Injecting this latter inequality in \eqref{eq:estimFsub1D}, we get $-\underline{F}''(x-x_0) \leq \mu_F (F^*(x-x_0) - \underline{F})$ for $x>x_0$; then in the same spirit as in Lemma \ref{lem:2}, we deduce for $x>x_0$,
  $$
  \underline{F}(x-x_0) \leq F^* (1-e^{-\sqrt{\mu_F}(x-x_0)}),
  $$
  which is equivalent to
  $$
  x-x_0 \geq -\frac{1}{\sqrt{\mu_F}} \ln \left(1-\frac{\underline{F}(x-x_0)}{F^*}\right).
  $$
  Injecting this latter inequality into \eqref{ineqM1D}, we obtain for $x>x_0$
  $$
  \underline{M}(x-x_0) \geq \frac{(1-\rho)\nu_E \underline{E}(x-x_0)}{2\mu_M}\left(1-\exp\left(2\frac{\sqrt{\mu_M}}{\sqrt{\mu_F}} \ln \left(1-\frac{\underline{F}(x-x_0)}{F^*}\right)\right)\right) = \phi(\underline{F}(x-x_0)).
  $$
  Injecting this latter inequality into \eqref{eq:estimFsub1D}, we get, for $x>x_0$,
  $$
  \partial_t \underline{\mathbf{F}} - \partial_{xx} \underline{\mathbf{F}} \leq \frac{\rho\nu_E b \underline{F}(x-x_0)}{\frac{b\underline{F}(x-x_0)}{K}+\mu_E+\nu_E}\Gamma(\underline{M}(x-x_0)) - \mu_F \underline{F}(x-x_0).
  $$
  Thus, $\underline{\mathbf{F}}$ is a subsolution in $\{x>x_0\}$ and since it is nondecreasing it verifies the condition at the interface $x=x_0$. 
\end{proof}

\noindent\textit{Proof of Proposition \ref{prop:TWMs0}. }
Since $\gamma>\gamma_c$, Lemma \ref{lem:stat} implies that there are two stable nonnegative steady states.
Existence of traveling waves follows then straightforwardly the work of \cite{FAN,ANG}. 
The fact that $\gamma\mapsto c_\gamma$ is increasing is a consequence of the fact that $\gamma\mapsto \Gamma$ is increasing. Indeed, if $\gamma_1<\gamma_2$ the traveling wave solution for $\gamma_1$ is clearly a subsolution of the system for $\gamma_2$.

The existence of a $\gamma_0$ such that \eqref{eq:condition} holds is a consequence of the fact that the left hand side of \eqref{eq:condition} is increasing with respect to $\gamma$, is negative for $\gamma=0$ and positive when $\gamma\to +\infty$ as $\mathcal{N}>1$ (see Remark \ref{rem:condition} above). 
It is proved in the Appendix that for $\gamma=\gamma_c$, we have for all $F\in (0,F^*)$ (see \eqref{rel_stat4}),
$$
\frac{\rho \nu_E b }{\frac{b F}{K} + \nu_E+\mu_E} \Gamma(\phi_0(F)) \leq \mu_F.
$$
Hence, $\gamma_0>\gamma_c$.

Finally, we are left to study the sign of the traveling wave. 
To do so, we use Proposition \ref{prop:sub1D} : for $\gamma_0$ such that \eqref{eq:condition} holds, there exists a subsolution $(\underline{\mathbf{E}},\underline{\mathbf{M}}, \underline{\mathbf{F}})$ of system \eqref{steady}. In particular, since the traveling wave is a solution, it should be bounded from below by this subsolution which is stationary. Necessarily, we have $c\geq 0$. 
Then, we conclude that for any $\gamma>\gamma_0$, we have $c>0$.
\qed

\subsection{Numerical illustration}\label{sec:num1d}

In order to illustrate the results in Proposition \ref{prop:TWMs0}, we display in this part some numerical results. We discretize system \eqref{syst0} in a one dimensional interval $[-L,L]$ for a time interval $[0,T]$ by a uniform semi-implicit $P_1$ finite element method, where the reaction term is treated explicitely. We take the numerical values given in Table \ref{tab:parametre} for the parameters of the model (these values are taken from \cite{STR19}). In this table, there is a wide range of choice for the parameter $\gamma$. Finally to fix the domain, we take $L=40$ and $T=150$. The initial data are chosen to be $(E^0,M^0,F^0) = (E^*,M^*,F^*) \mathbf{1}_{x<-10}$ such that the initial data is at the positive stable equilibrium on the left of the domain and at the zero stable equilibrium at the right.
\begin{table}[h!]
    \centering
    \begin{tabular}{|c|c|c|c|c|c|c|c|c|}
       \hline 
       $b$ & $\nu_E$ & $\mu_E$ & $\mu_M$ & $\mu_F$ & $\rho$ & $K$ & $D$ & $\gamma$ \\
       \hline  
       10  & 0.08 & 0.05 & 0.14 & 0.1 & 0.5 & 200 & $0.1$ & $10^{-4} - 1$ \\
       \hline
    \end{tabular}
    \caption{Numerical values of the parameters of the model.}
    \label{tab:parametre}
\end{table}
With the numerical values in Table \ref{tab:parametre}, we first consider the value $\gamma=0.5$. Then, we compute $F^* = 77.4$ and we find the numerical values $\gamma_c = 2.351\times 10^{-3}$ and $\gamma_0 = 4,3\times 10^{-2}$.
Hence, we are in the situation where $\gamma_c<\gamma_0<\gamma$ for which the results of Proposition \ref{prop:TWMs0} apply. The numerical results are shown in Figure \ref{fig1}. As expected, we observe a traveling wave with positive speed illustrating the fact that there is an invasion of the species into the domain. 
\begin{figure}[h!]
    \centering
    \includegraphics[width=0.5\linewidth]{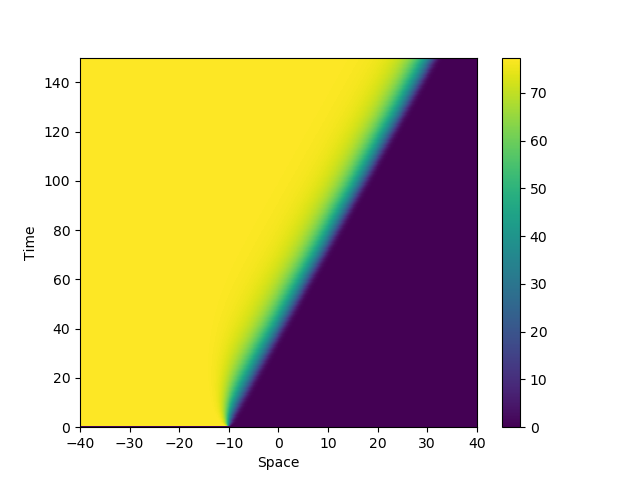}
    \caption{Time and spatial dynamics of the density of female mosquitoes $F$ solution of \eqref{syst0} with the numerical parameters in Table \ref{tab:parametre} in the situation of Proposition \ref{prop:TWMs0} where $\gamma_c<\gamma_0<\gamma$. We observe invasion of the species into the domain.}
    \label{fig1}
\end{figure}
We also show in Figure \ref{fig2} two situations where the conditions of Proposition \ref{prop:TWMs0} are not fulfilled. In Figure \ref{fig2} left, we take $\gamma=0.01$. Then, we find $F^*=30.12$, $\gamma_c = 2.351\times 10^{-3}$ and $\gamma_0 = 1,5\times 10^{-2}$. Hence, we are in the situation $\gamma_c<\gamma<\gamma_0$, however we still observe an invasion of the mosquito population into the domain. It illustrates the fact that condition $\gamma>\gamma_0$ is not optimal. Nevertheless, for $\gamma$ even smaller, we observe that there may be no invasion of the mosquito population (see Fig. \ref{fig2}-right where we took $\gamma = 2.355\times 10^{-3}$). Obviously in this latter situation there is no need to apply the sterile insect technique.
\begin{figure}[h!]
    \centering
    \includegraphics[width=0.45\linewidth]{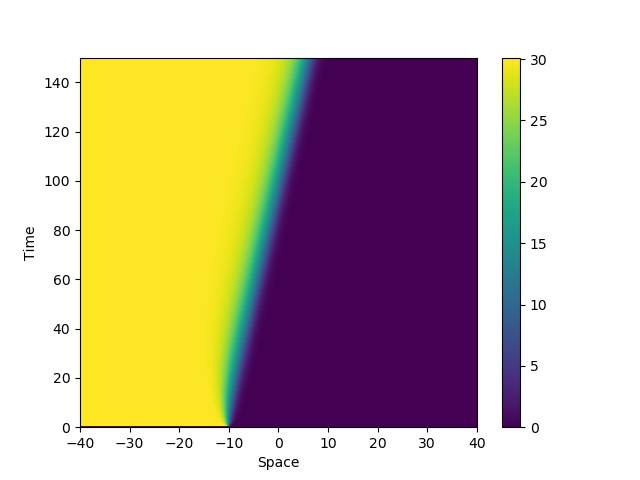}
    \includegraphics[width=0.45\linewidth]{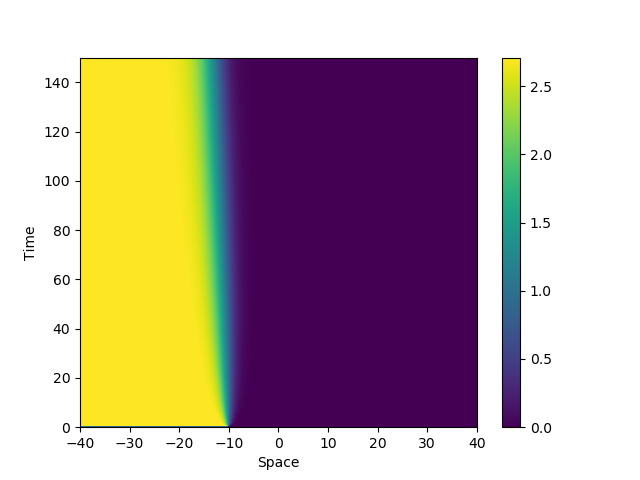}
    \caption{Time and spatial dynamics of the density of female mosquitoes $F$, solution of \eqref{syst0}, in the situation where $\gamma_c<\gamma<\gamma_0$ for $\gamma=0.01$ (left) and $\gamma=2.355\, 10^{-3}$ (right). We observe that we may have invasion (left) or natural extinction (right) of the mosquito population.}
    \label{fig2}
\end{figure}

\section{Construction of a radially symmetric sub-solution with sterile males}\label{sec:sub-solution}

As presented in the introduction, we use similar arguments as in the case without sterile males. However, a main new difficulty arises: the equation is non-autonomous because of the sterile males. The idea is as follows: since the sub-solution becomes nonzero (i.e. $(\underline{E}, \underline{F}, \underline{M}) > (0,0,0)$) for large values of $|x - ct|$, we may assume that, in this region, the density of sterile males is negligible (of order $\varepsilon e^{-|x - ct|}$) and thus the previous techniques can be applied.

\subsection{A stationary problem in a half space}

Let $\varepsilon>0$, following the strategy developed in Section \ref{sec:Ms0}, we investigate stationary solutions of the following problem on $(0,+\infty)$,
\begin{subequations}\label{stat2}
\begin{align}\label{stat2:M}
  & - M'' = g_M(M,F) := \frac{(1-\rho) \nu_E b F}{\frac{bF}{K}+\mu_E+\nu_E} - \mu_M M  \\
  \label{stat2:F}
  & - F'' = g_F^\varepsilon(x,M,F) := \frac{\rho \nu_E b F}{\frac{bF}{K}+\mu_E+\nu_E}\frac{M}{M+m_s^\varepsilon(x)} \Gamma(M) - \mu_F F,
\end{align}
\end{subequations}
where $m_s^\varepsilon(x) = \varepsilon e^{-\sqrt{\mu_s}x}$, complemented with the initial condition $(M(0),F(0)) = (0,0)$.

We will use the function $\phi$ defined in \eqref{def:phi}, and we also introduce
\begin{equation}
    \label{def:phis}
\phi_s^\varepsilon(F):= \varepsilon \exp\left(\sqrt{\dfrac{\mu_s}{\mu_F}}\ln \left(1-\frac{F}{F^*}\right)\right).
\end{equation}
Then, we first consider the following scalar reaction-diffusion equation on $(0,+\infty)$
\begin{equation}\label{eq:F2}
-F'' = \frac{\rho \nu_E b F}{\frac{b F}{K} + \mu_E + \nu_E} \frac{\phi(F)}{\phi(F)+\phi_s^\varepsilon(F)}\Gamma(\phi(F)) - \mu_F F.
\end{equation}
complemented with initial data $\underline{F}(0)=0$
\begin{proposition}\label{prop:sub2}
Under the assumptions of Proposition \ref{prop:exist}, there exists $\varepsilon_0>0$ such that for any $0<\varepsilon<\varepsilon_0$, there exists a solution $\underline{F}$ of \eqref{eq:F2} on $(0,+\infty)$ with $\underline{F}(0)=0$, which is bounded, non-decreasing and there exists $\underline{F_n}>F_1^*$ such that $\lim_{x\to +\infty} \underline{F}(x) = \underline{F_n}$
\end{proposition}
\begin{proof}
We first observe that, as a consequence of the dominated convergence theorem, we have
\begin{align*}
    \lim_{\varepsilon\to 0}
    \int_0^{F^*} \left(\frac{\rho\nu_E b u}{\frac{b u}{K} + \mu_E + \nu_E}\frac{\phi(u)}{\phi(u)+\phi_s^\varepsilon(u)} \Gamma(\phi(u)) - \mu_F u \right)\,du  \\
    = \int_0^{F^*} \left(\frac{\rho\nu_E b u}{\frac{b u}{K} + \mu_E + \nu_E} \Gamma(\phi(u)) - \mu_F u \right)\,du.
\end{align*}
Hence, from \eqref{condition}, we deduce that there exists $\varepsilon_0$ small enough such that, for $0<\varepsilon<\varepsilon_0$, we have
    $$
    G_\varepsilon(F):= \int_0^{F^*} \left(\frac{\rho\nu_E b u}{\frac{b u}{K} + \mu_E + \nu_E}\frac{\phi(u)}{\phi(u)+\phi_s^\varepsilon(u)} \Gamma(\phi(u)) - \mu_F u \right)\,du>0.
    $$
    As in the proof of Proposition \ref{prop:exist}, we may construct the solution by taking $\underline{F_n}\in (0,F^*]$ such that $G_\varepsilon(\underline{F_n}):=\max_{[0,F^*]} G_\varepsilon$ and solving the Cauchy problem
    $$
    \underline{F}' = \sqrt{2\left(G_\varepsilon(\underline{F_n})-G_\varepsilon(\underline{F}))\right) }, \qquad \underline{F}(0) = 0.
    $$
    Clearly this solution is nondecreasing.
    And by the same token as in the proof of Proposition \ref{prop:exist}, we verify that $\lim_{x\to +\infty} \underline{F}(x)=\underline{F_n}>F_1^*$.
\end{proof}

\begin{proposition}\label{prop:sol}
    Let $\gamma>\gamma_0$ where $\gamma_0$ is defined in \eqref{eq:condition}. 
    Then, there exists $\varepsilon_\gamma>0$ and a bounded solution $(\mathbf{M}_\gamma,\mathbf{F}_\gamma)$ of \eqref{stat2} with $\varepsilon=\varepsilon_\gamma$ in the definition of $g_F^\varepsilon$, and with $(\mathbf{M}_\gamma(0),\mathbf{F}_\gamma(0))=(0,0)$ and $(\mathbf{M}_\gamma(+\infty),\mathbf{F}_\gamma(+\infty))=(M^*,F^*)$.
    Moreover, $(\mathbf{M}_\gamma,\mathbf{F}_\gamma)$ is nondecreasing.
\end{proposition}
\begin{proof} We split the proof into several steps~:

\textbf{Step 1:} Construction of a super-solution and a sub-solution.

    On the one hand, let us denote, for $x\in(0,+\infty)$
    $$
    \overline{M}(x) := M^*(1-e^{-\sqrt{\mu_M}x}), \qquad \overline{F}(x) := F^*(1-e^{-\sqrt{\mu_F}x})).
    $$
    We have 
\[    -\overline{M}'' = \mu_M M^* - \mu_M \overline{M} = \frac{(1-\rho) \nu_E b F^*}{\frac{b F^*}{K}+\mu_E+\nu_E} - \mu_M \overline{M}  \geq g_M(\overline{M},\overline{F}),
\]
since $\overline{F}\leq F^*$ and $g_M$ is increasing with respect to its second variable.
    For the second equation,
\[    -\overline{F}'' = \mu_F F^* - \mu_F \overline{F} 
    = \frac{\rho \nu_E b F^*}{\frac{b F^*}{K}+\mu_E+\nu_E} \Gamma(M^*) - \mu_F \overline{F}  \geq g_F^\varepsilon(x,\overline{M},\overline{F}),\]    since $\overline{M}\leq M^*$ and $\overline{F}\leq F^*$ and the first term of the right hand side is increasing with respect to $M$ and to $F$.
    Thus, $(\overline{M},\overline{F})$ is a super-solution for \eqref{stat2}.

    On the other hand, from Proposition \ref{prop:sub2}, there exists $\varepsilon_\gamma$ small enough, such that there exists a solution $\underline{F}$ of \eqref{eq:F2} with $\varepsilon=\varepsilon_\gamma$ in the definition of $\phi_s^\varepsilon$ (see \eqref{def:phis}). Then, with this function $\underline{F}$, we define $\underline{M}$ solution of
    $$
    -\underline{M}'' = g_M(\underline{M},\underline{F}), \qquad \underline{M}(0) = 0, \quad \underline{M}(+\infty) = \underline{M_n} := \dfrac{(1-\rho) \nu_E b \underline{F_n}}{\mu_M\left(\frac{b \underline{F_n}}{K}+\mu_E+\nu_E\right)}.
    $$
    From Lemma \ref{lem:1}, such a solution $\underline{M}$ exists and verifies
    \begin{equation}\label{estim:Mbar}
    \underline{M}(x) \geq \frac{1}{2\mu_M}\dfrac{(1-\rho) \nu_E b \underline{F}}{\frac{b \underline{F}}{K}+\mu_E+\nu_E}\left(1-e^{-2\sqrt{\mu_M}x}\right).
    \end{equation}
    Then, we claim that $(\underline{M},\underline{F})$ is a sub-solution for \eqref{stat2}.

    In order to prove this claim, it suffices to show that $\underline{F}$ is a sub-solution for \eqref{stat2:F}. Indeed, the first term of the right hand side of \eqref{stat2:M} is increasing with respect to $F$.
    With the definitions of $\phi$ in \eqref{def:phi} and $\phi_0$ in \eqref{def:phi0}, it is clear that $\phi(F)<\phi_0(F)$. Recalling moreover (see \eqref{rel_stat2} in Appendix) that
    $$
    \frac{\rho \nu_E b F^*}{\frac{b F^*}{K}+\mu_E+\nu_E} \Gamma(\phi_0(F^*)) = \mu_F F^*,
    $$
    we deduce from \eqref{eq:F2} that $-\underline{F}'' \leq \mu_F (F^*-\underline{F})$. Lemma \ref{lem:1} implies that $\underline{F}\leq F^*(1-e^{-\sqrt{\mu_F}x})$ for any $x\in (0,+\infty)$, or equivalently
    $$
    x \geq - \frac{1}{\sqrt{\mu_F}} \ln \left(1-\frac{\underline{F}}{F^*}\right).
    $$
    Then, by definition of $\phi$ in \eqref{def:phi} and $\phi_s^\varepsilon$ in \eqref{def:phis}, we deduce, using also \eqref{estim:Mbar} that
    $$
    \phi(\underline{F}) \leq \underline{M}
    \qquad \text{ and } \qquad
    \varepsilon_\gamma e^{-\sqrt{\mu_s} x} \leq  \phi_s^{\varepsilon_\gamma}(\underline{F}) .
    $$
    Inserting into \eqref{eq:F2}, we obtain 
    $$
    -\underline{F}'' \leq g_F^{\varepsilon_\gamma}(x,\underline{M},\underline{F}).
    $$

    Then, we have constructed a super- and a sub-solution for system \eqref{stat2}. It is then classical to construct a solution, denoted $(\mathbf{M}_\gamma,\mathbf{F}_\gamma)$ such that $(\underline{M},\underline{F}) \leq (\mathbf{M}_\gamma,\mathbf{F}_\gamma) \leq (\overline{M},\overline{F})$.

\medskip    

\textbf{Step 2:} $(\mathbf{M}_\gamma,\mathbf{F}_\gamma)$ is non-decreasing.

    Assume by contradiction that there exists a point $x_0>0$ and $\delta_0>0$ such that $\mathbf{M}_\gamma'(x_0) = 0$ and $\mathbf{M}_\gamma'(x)<0$ for $x\in(x_0,x_0+\delta_0)$ (the proof is similar if the monotony is first broken by the function $\mathbf{F}_\gamma$). Then, either there exists $x_1\geq x_0$ and $\delta_1>0$ such that $\mathbf{F}_\gamma'(x_1)=0$ and $\mathbf{F}'_\gamma(x)<0$ on $(x_1,x_1+\delta_1)$, or $\mathbf{F}_\gamma$ is non-decreasing. 
    We define
$$
\widetilde{\mathbf{M}}(x) = \left\lbrace  \begin{aligned} & \mathbf{M}_\gamma(x) && \text{ for } x < x_0,\\ &\max(\mathbf{M}_\gamma(x), \mathbf{M}_\gamma(x_0)) && \text{ for } x \geq x_0, \end{aligned} \right. 
$$
and if $x_1<+\infty$,
$$
\widetilde{\mathbf{F}}(x) = \left\lbrace  \begin{aligned} & \mathbf{F}_\gamma(x) && \text{ for } x < x_1,\\ &\max(\mathbf{F}_\gamma(x), \mathbf{F}_\gamma(x_1)) && \text{ for } x \geq x_1, \end{aligned} \right.
$$
else, $\widetilde{\mathbf{F}} = \mathbf{F}_\gamma$.
Clearly, we have by definition $(\mathbf{M}_\gamma,\mathbf{F}_\gamma)\leq (\widetilde{\mathbf{M}},\widetilde{\mathbf{F}})$ and $(\widetilde{\mathbf{M}},\widetilde{\mathbf{F}})$ is non-decreasing.
Next, we claim that $(\widetilde{\mathbf{M}},\widetilde{\mathbf{F}})$ is a sub-solution which will be a contradiction. 
It is clear that the claim is true for $x<x_0$ since it is a solution. Therefore, we focus on the set $\{x\geq x_0\}$: 
\begin{itemize}
    \item \textit{Equation on $M$. } We distinguish two cases~: 
    
    If $\widetilde{\mathbf{M}}(x) = \mathbf{M}_\gamma(x)$ with $\widetilde{\mathbf{M}}''(x) \neq 0$, we have
    $$ 
    -\widetilde{\mathbf{M}}''(x) - g_M(\widetilde{\mathbf{M}}, \widetilde{\mathbf{F}})(x) = (1-\rho) \nu_E b  \left( \frac{\mathbf{F}_\gamma(x)}{\frac{b}{K} \mathbf{F}_\gamma(x) + \mu_E + \nu_E} - \frac{\widetilde{\mathbf{F}}(x)}{\frac{b}{K}\widetilde{\mathbf{F}}(x) + \mu_E + \nu_E} \right) \leq 0,
    $$
    since $\mathbf{F}_\gamma \leq  \widetilde{\mathbf{F}}$ by definition.

    If $\widetilde{\mathbf{M}}(x) = \mathbf{M}_\gamma(x_0)$ with $\widetilde{\mathbf{M}}''(x) = 0$. Using that $\mathbf{F}_\gamma(x_0) \leq \widetilde{\mathbf{F}}(x)$ since $x>x_0$ and by definition of $\widetilde{\mathbf{F}}$, we have 
    $$ 
    -\widetilde{\mathbf{M}}''(x) - g_M(\widetilde{\mathbf{M}}, \widetilde{\mathbf{F}})(x) =  - g_M(\mathbf{M}_\gamma(x_0),\widetilde{\mathbf{F}}(x)) \leq   - g_M(\mathbf{M}_\gamma(x_0),\mathbf{F}_\gamma(x_0)) = \mathbf{M}_\gamma''(x_0) \leq 0, 
    $$
    since, by definition, $x_0$ is a local maximum for $\mathbf{M}_\gamma$.
    \item \textit{Equation on $F$.} Similarly, we have~:

    If $\widetilde{\mathbf{F}}(x) = \mathbf{F}_\gamma(x)$, we have
    $$
    -\widetilde{\mathbf{F}}''(x) - g_F^\varepsilon(x,\widetilde{\mathbf{M}}(x),\widetilde{\mathbf{F}}(x))  = g_F^\varepsilon(x,\mathbf{M}_\gamma(x),\mathbf{F}_\gamma(x)) - g_F^\varepsilon(x,\widetilde{\mathbf{M}}(x),\widetilde{\mathbf{F}}(x)) \leq 0,
    $$
    since $\mathbf{M}_\gamma \leq  \widetilde{\mathbf{M}}$ by definition and $g_F^\varepsilon$ is non-decreasing with respect to $M$.

    If $\widetilde{\mathbf{F}}(x) = \mathbf{F}_\gamma(x_1)$ for $x>x_1$ with $\widetilde{\mathbf{F}}''(x) = 0$, we compute
    $$
     -\widetilde{\mathbf{F}}''(x) - g_F^\varepsilon(x,\widetilde{\mathbf{M}}(x),\widetilde{\mathbf{F}}(x))  = - g_F^\varepsilon(x,\widetilde{\mathbf{M}}(x),\mathbf{F}_\gamma(x_1)) \leq - g_F^\varepsilon (x_1,\mathbf{M}_\gamma(x_1),\mathbf{F}_\gamma(x_1))  = \mathbf{F}''_\gamma(x_1),
    $$
    since $x>x_1$ and $\widetilde{\mathbf{M}}(x)\geq \widetilde{\mathbf{M}}(x_1) \geq \mathbf{M}_\gamma(x_1)$. Moreover, since $x_1$ is a local maximum for $\mathbf{F}_\gamma$, we have 
    $\mathbf{F}''_\gamma(x_1)\leq 0$. Hence,
    $$
     -\widetilde{\mathbf{F}}''(x) - g_F^\varepsilon(x,\widetilde{\mathbf{M}}(x),\widetilde{\mathbf{F}}(x)) \leq 0.
     $$
\end{itemize}
We conclude the proof of this step by stating that we have constructed a new sub-solution which is greater than the solution. This is in contradiction with the definition of a solution.

\medskip    
\textbf{Step 3.} The limit $\lim_{x\to +\infty} (\mathbf{M}_\gamma,\mathbf{F}_\gamma)(x) = (M^*,F^*)$.    

    In the previous step, we have proved that $\mathbf{M}_\gamma$ and $\mathbf{F}_\gamma$ are non-decreasing and are bounded by $M^*$ and $F^*$ respectively. Hence, they converge to some limit which is a non-trivial steady state of system \eqref{stat2}. Moreover, from Proposition \ref{prop:sub2}, the limit for the function $\mathbf{F}_\gamma$ should be greater than $\underline{F_n}>F_1^*$, and similarly the limit for the function $\mathbf{M}_\gamma$ should be greater than $\underline{M_n}>M_1^*$. Since $M_s \to 0$, the only steady state greater than $(M_1^*,F_1^*)$ is $(M^*,F^*)$.
\end{proof}

\subsection{Construction of a subsolution}

We start by an estimate on the sterile male density.
\begin{lemma}\label{lem:ms}
    Let $c>0$ and $\Lambda$ be as in the statement of Theorem \ref{MainTheorem}.
    Let us assume that $M_s^0\in L^\infty(\RR^2)$ is compactly supported in a ball of radius $R_s^0$. 
    Then, there exists $R_s$ large enough such that the solution of \eqref{syst:Ms} on $\RR^2$ with initial data $M_s^0$ verifies 
    $$
    M_s(t,x) \leq \overline{M_s}(t,x) := \max\left(\|M_s^0\|_\infty,\dfrac{\bar{\Lambda}}{\mu_s}\right) \left(\mathbf{1}_{|x|\leq R_s+ct} + e^{-\sqrt{\mu_s}(|x|-R_s-ct)} \mathbf{1}_{|x|>R_s+ct} \right).
    $$
\end{lemma}
\begin{proof}
    From the definition of $\Lambda$ in \eqref{Lambda}, we have
    $$
    \Lambda(x,t) \leq \bar{\Lambda} \mathbf{1}_{|x|\leq R_2+ct}.
    $$
    For $R_s>\max(R_2,R_s^0)$, we verify that $\overline{M_s}$ is a super-solution for the equation on $M_s$.
    We first have
    $$
    \partial_t \overline{M_s} - \Delta \overline{M_s} - \Lambda + \mu_s \overline{M_s}
    \geq -\left(c+\dfrac{1}{r}\right) \partial_r \overline{M_s} - \partial_{rr} \overline{M_s} - \bar{\Lambda} \mathbf{1}_{|x|\leq R_2+ct} + \mu_s \overline{M_s}.
    $$
    For $|x|<R_s+ct$, $\overline{M_s}$ is a constant such that $\mu_s \overline{M_s} \geq \bar{\Lambda}$. Hence, we deduce from above inequality that
    $$
    \partial_t \overline{M_s} - \Delta \overline{M_s} - \Lambda + \mu_s \overline{M_s}
    \geq 0.
    $$
    For $|x|>R_s+ct$, we compute
    $$
    -\left(c+\dfrac{1}{r}\right) \partial_r \overline{M_s} - \partial_{rr} \overline{M_s} - \bar{\Lambda} \mathbf{1}_{|x|\leq R_2+ct} + \mu_s \overline{M_s}
    = \sqrt{\mu_s}\left(c+\dfrac{1}{r}\right) \overline{M_s} - \mu_s \overline{M_s} + \mu_s \overline{M_s} \geq 0.
    $$
    At the interface $|x|=R_s+ct$, we have easily
    $$
    \lim_{|x|\to (R_s+ct)^-} \partial_x \overline{M_s} = 0 \geq \lim_{|x|\to (R_s+ct)^+}\partial_x \overline{M_s}.
    $$
    Finally, by definition we also have $\overline{M_s}(t=0,x) \geq M_s^0$. Hence, $\overline{M_s}$ is a super-solution and this concludes the proof.
\end{proof}

Using the function $(\mathbf{M}_\gamma,\mathbf{F}_\gamma)$ of Proposition \ref{prop:sol}, we may construct a sub-solution for system \eqref{syst}.
\begin{proposition}\label{prop:subsolsyst}
    Let $c>0$, $\gamma>\gamma_0$ and $\Lambda$ be as in the statement of Theorem \ref{MainTheorem}. 
    Let $(\mathbf{M}_\gamma,\mathbf{F}_\gamma)$ be given by Proposition \ref{prop:sol}.
    There exists $\underline{R}$ large enough such that $(\underline{\mathbf{E}},\underline{\mathbf{M}},\underline{\mathbf{F}},\underline{\mathbf{M_s}})$ defined by
    \begin{align*}
    & \underline{\mathbf{M}}(t,x) = \mathbf{M}_\gamma(|x|-ct-\underline{R}) \mathbf{1}_{|x|>ct+\underline{R}}, \qquad
    \underline{\mathbf{F}}(t,x) = \mathbf{F}_\gamma(|x|-ct-\underline{R}) \mathbf{1}_{|x|>ct+\underline{R}},  \\
    & \underline{\mathbf{E}} = \frac{b \underline{\mathbf{F}}}{\frac{b}{K} \underline{\mathbf{F}} + \mu_E+\nu_E}, \qquad 
    \underline{\mathbf{M_s}}(t,x) = \overline{M_s}(t,x),
    \end{align*}
    with $\overline{M_s}$ defined in the statement of Lemma \ref{lem:ms}, is a sub-solution for the system \eqref{syst} with initial conditions verifying \eqref{hyp:init2}. 
\end{proposition}
\begin{proof}
Let $\gamma>\gamma_0$. From Proposition \ref{prop:sol}, there exists $\varepsilon_\gamma$ and $(\mathbf{M}_\gamma,\mathbf{F}_\gamma)$. 
Then, $(\underline{\mathbf{E}},\underline{\mathbf{M}}, \underline{\mathbf{F}})$ is well-defined, continuous on $\RR^2$, radially non-decreasing and it converges to $(E^*,M^*,F^*)$ as $r$ goes to $+\infty$ thanks to Proposition \ref{prop:sol}.
Moreover, by definition of $\overline{M_s}$ in Lemma \ref{lem:ms}, there exists $\underline{R}>R_s$ large enough such that for $|x|\geq \underline{R}+ct$, we have 
\begin{equation}\label{estim:Msbar}
\overline{M_s}(t,x)\leq m_s^{\varepsilon_\gamma}(|x|-\underline{R}-ct) = \varepsilon_\gamma e^{-\sqrt{\mu_s}(|x|-\underline{R}-ct)}.
\end{equation}

We check that $(\underline{\mathbf{E}},\underline{\mathbf{M}},\underline{\mathbf{F}},\underline{\mathbf{M_s}})$ is a sub-solution of each equation of \eqref{syst} separately. 

\noindent \textit{Equation on $E$}. This is the simplest one. Indeed, we have by definition
  $$
  \partial_t \underline{\mathbf{E}} = -c \underline{\mathbf{E}}' \leq 0 = b \underline{\mathbf{F}} (1-\frac{\underline{\mathbf{E}}}{K}) - (\mu_E+\nu_E) \underline{\mathbf{E}},
  $$
  since $\underline{\mathbf{E}}$ is radially nondecreasing.
  
  \medskip
  
 \noindent  \textit{Equation on $M$}. For $r<\underline{R}+ct$ it is clear. For $r>\underline{R}+ct$, we compute using the fact that $\mathbf{M}_\gamma$ is nondecreasing
  \begin{align*}
  \partial_t \underline{\mathbf{M}} - \Delta \underline{\mathbf{M}} 
  & = - \mathbf{M}_\gamma''(|x|-ct-\underline{R}) -(c+\frac{1}{r}) \mathbf{M}_\gamma'(|x|-ct-\underline{R}) \\
  & \leq -\mathbf{M}_\gamma''(|x|-ct-\underline{R}) = (1-\rho)\nu_R \underline{\mathbf{E}} - \mu_M \underline{\mathbf{M}},
  \end{align*}
  where we use \eqref{stat2:M} for the last equality. Thus $\underline{\mathbf{M}}$ is a sub-solution for $r>\underline{R}+ct$. Since $\partial_r \underline{\mathbf{M}}((\underline{R}+ct)^+) = \mathbf{M}_\gamma'(0^+)\geq 0 = \partial_r \underline{\mathbf{M}}((\underline{R}+ct)^-)$, it is also a sub-solution on the whole domain $\RR^2$.
  
  \medskip

  \noindent \textit{Equation on $F$}. For $r<\underline{R}+ct$ it is clear. For $r>\underline{R}+ct$, we compute, using the fact that $\mathbf{F}_\gamma$ is nondecreasing
  \begin{align*}
    \partial_t \underline{\mathbf{F}} - \Delta \underline{\mathbf{F}}
    & = -\mathbf{F}_\gamma''(|x|-ct-\underline{R}) - (c+\frac{1}{r}) \mathbf{F}_\gamma'(|x|-ct-\underline{R}) \\
    & \leq -\mathbf{F}_\gamma''(|x|-ct-\underline{R}) = g_F^{\varepsilon_\gamma}(|x|-ct-\underline{R},\mathbf{M}_\gamma(|x|-ct-\underline{R}),\mathbf{F}_\gamma(|x|-ct-\underline{R})) \\
    & \leq \frac{\rho \nu_E b \underline{\mathbf{F}}}{\frac{b}{K} \underline{\mathbf{F}} + \mu_E + \nu_E} \dfrac{\underline{\mathbf{M}}}{\underline{\mathbf{M}}+m_s^{\varepsilon_\gamma}(|x|-ct-\underline{R})} \Gamma(\underline{\mathbf{M}}) - \mu_F \underline{\mathbf{F}}  \\
    & \leq \frac{\rho \nu_E b \underline{\mathbf{F}}}{\frac{b}{K} \underline{\mathbf{F}} + \mu_E + \nu_E} \dfrac{\underline{\mathbf{M}}}{\underline{\mathbf{M}} + \overline{M_s}} \Gamma(\underline{\mathbf{M}}) - \mu_F \underline{\mathbf{F}},
  \end{align*}
  where we have used \eqref{stat2:F} and \eqref{estim:Msbar}.

  Thus, $\underline{\mathbf{F}}$ is a subsolution on $r>\underline{R}+ct$ and since it is nondecreasing it verifies the condition at the interface $r=\underline{R}+ct$.

  \textit{Equation on $M_s$}. The conditions on $\underline{\mathbf{M_s}}$ have already been verified in Lemma \ref{lem:ms}.

  Finally, thanks to assumption \eqref{hyp:init2} the conditions on the initial data are verified easily provided $\underline{R}$ is large enough. This concludes the proof.    
\end{proof}

%%%%%%%%%%%%%%%%%%%%%%%%%%%%%%%

%%%%%%%%%%%%%%%%%%%%%%%%%%%%%%%%%%%%%%

\section{Construction of a radially symmetric super-solution} \label{sec:super-solution}

\subsection{Technical lemma}

\begin{lemma}\label{lem:phi1}
  Let $\mu>0$, $r_1>0$, and $u_0>0$, and let $0<\underline{c}<c'<c$ with $c'> \frac 23 c$.
  Let us define
  \begin{equation}\label{def:alphabeta}
    \alpha(t) := \frac{u_0}{\lambda_+ e^{\lambda_-[c- c']t}- \lambda_-e^{\lambda_+[c- c']t}},
    \qquad \beta(r) := \lambda_+ e^{\lambda_-r}- \lambda_-e^{\lambda_+ r}.
  \end{equation}
  where $\displaystyle \lambda_{\pm} = \frac{ - (c' + \frac{1}{r_1}) \pm \sqrt{(c' + \frac{1}{r_1})^2 + 2\mu }}{2}$. Then, the following hold~:
  \begin{itemize}
  \item[(i)] The function $t\mapsto \alpha(t)$ is positive and decreasing for $t\geq 0$, $\displaystyle \lim_{t\to +\infty} \alpha(t) = 0$, and for all $t\geq 0$, we have $\alpha'(t) > -\dfrac{\mu}{4} \alpha(t)$.
  \item[(ii)] The function $r\mapsto \beta(r)$ is positive and increasing for $r\geq 0$. Moreover for all $r>0$, we have $\beta'(r) <  \sqrt{\dfrac{\mu}{2}} \beta(r)$.
  \item[(iii)] The function $\phi_1$ defined by
  \begin{equation}\label{eq:phi1}
    \phi_1(x,t) = \alpha(t) \beta(|x|-(r_1+c't)),
  \end{equation}
  is a super-solution on $\Omega^1_t$ of the equation $\partial_t u - \Delta u = -\mu u$ which verifies the Dirichlet condition $u=u_0$ on the boundary $\{|x|=r_1+ct\}$ and Neumann condition $\partial_r u = 0$ on the boundary $\{|x|=r_1+c't\}$.
\end{itemize}
\end{lemma}
\begin{remark}\label{rem:psi1}
Notice that for $t>0$ fixed, the function $\psi_1(r) := \alpha(t)\beta(r-(r_1+c't))$ is a solution of the boundary value problem
  \[ \left\lbrace 
      \begin{aligned}
        &-(c' + \frac{1}{r_1}) \psi_1' - \psi_1'' = - \frac{\mu }{2} \psi_1 \\
        &\psi_1(r_1 + ct) = u_0, \\
        &\psi_1'(r_1 + c't) = 0. 
      \end{aligned}
    \right.
  \]
\end{remark}

\begin{proof}
  For the point $(i)$, we notice that by definition of $\lambda_\pm$, we have $0<\lambda_+<-\lambda_-$. Next, $\alpha$ is clearly positive and goes to $0$ as $t$ grows to $+\infty$.
  Then, we compute
  \[ \alpha'(t) = -\frac{u_0\lambda_+\lambda_-[c-c']\left(e^{\lambda_-[c- c']t}- e^{\lambda_+[c- c']t}\right)}{(\lambda_+ e^{\lambda_-[c- c']t}- \lambda_-e^{\lambda_+[c- c']t})^2} = [c - c'] \cdot \frac{\lambda_+\lambda_-\left(e^{\lambda_+[c- c']t} - e^{\lambda_-[c- c']t}\right)}{\lambda_+ e^{\lambda_-[c- c']t}- \lambda_-e^{\lambda_+[c- c']t}} \alpha(t).\]
  By definition of $\lambda_\pm$, we have $\lambda_+\lambda_-= -\frac{\mu}{2}$.
  Moreover, since $0<\lambda_+<-\lambda_-$, we have 
$$
\frac{e^{\lambda_+[c- c']t} - e^{\lambda_-[c- c']t}}{\lambda_+ e^{\lambda_-[c- c']t}- \lambda_-e^{\lambda_+[c- c']t}}\leq \frac{e^{\lambda_+[c- c']t} - e^{\lambda_-[c- c']t}}{- \lambda_-e^{\lambda_+[c- c']t}} \leq -\frac{1}{\lambda_-} < \frac{1}{c'}.
$$
Then, for all $c'\in(\frac 23 c,c)$, which is equivalent to $\frac{c-c'}{c'}\in (0,\frac 12)$, we have
 \[\alpha'(t) > -\frac{\mu}{2} \frac{[c - c']}{c'} \alpha(t)> -\frac{\mu}{4} \alpha(t).\]

 For oint $(ii)$, the positivity of $\beta$ is clear since $\lambda_-<0<\lambda_+$. Then, we have
 \begin{align*}
   \beta'(r) = -\lambda_+\lambda_- (e^{\lambda_+ r} - e^{\lambda_- r}) > 0
 \end{align*}
 and
 \begin{align*}
   \beta'(r) = -\lambda_+\lambda_- (e^{\lambda_+ r} - e^{\lambda_- r}) < -\lambda_+\lambda_- e^{\lambda_+ r} < \lambda_+ \beta(r) < \sqrt{\dfrac{\mu}{2}} \beta(r).
 \end{align*}
 
  For oint $(iii)$, we first notice that we have
  $$
  -(c' + \frac{1}{r_1})\beta' - \beta'' + \frac{\mu}{2} \beta= 0.
  $$
Then, we compute, denoting $r=|x|$,
\[ 
\begin{aligned}
\partial_t \phi_1 - \Delta \phi_1 + \mu \phi_1  
&= \left( \alpha' + \frac{\mu}{2}\alpha \right)(t) \beta(r -[r_1 + c't]) + \alpha(t) \left[ -(c' + \frac{1}{r})\beta' - \beta'' + \frac{\mu}{2} \beta  \right](r- [r_1 +c't]) \\
&= \left( \alpha' + \frac{\mu}{2} \alpha \right)(t) \beta(r -[r_1 + c't])+\alpha(t) \left( \frac{1}{r_1} - \frac{1}{r}\right) \beta'(r-[r_1 +c't]).
\end{aligned}
\]
Recalling that $r_1 + c't < r < r_1 + ct$, $\alpha(t) > 0$ and $\beta$ is increasing, we have that 
$$
\alpha(t) \left( \frac{1}{r_1} - \frac{1}{r}\right) \beta'(r-[r_1 +c't])>0.
$$
Hence,
\begin{equation}\label{eq:a:T:1}
\partial_t \phi_1 - \Delta \phi_1 + \mu \phi_1 > \left( \alpha' + \frac{\mu}{2} \alpha \right)(t) \beta(r -[r_1 + c't]).
\end{equation}
Then, we use the inequality in point $(i)$ and deduce that, for all $c'\in(\frac 23 c,c)$, 
\begin{equation}
\label{eq:a:T:2}
(\alpha' + \frac{\mu}{2}\alpha ) \beta > \frac{\mu \alpha \beta}{4}>0. 
\end{equation}
We conclude thanks to \eqref{eq:a:T:1}.
\end{proof}

The next preliminary result is similar to \cite[Lemma 2]{ALM3}~:
\begin{lemma}\label{lem:psi}
  Let $\varepsilon>0$, $u_0\in (0,1)$, $c>0$, and $r_1>0$. 
  There exists a constant $L$, large enough, such that there exists a solution $\psi$ to the following system:
\[\left\lbrace 
\begin{aligned}
& -(c + \frac{1}{r_1}) \psi' - \psi'' = -\varepsilon \psi, \\
&\psi(0) = u_0, \quad \psi'(0) = 0, \\
&\psi(L) = 1, \quad \psi'(L) > 0.
\end{aligned}\right.\]
Moreover, $\psi$ is positive, increasing on $(0,L)$, and we have $0 < \psi'(r) < \sqrt{\varepsilon} \psi(r)$.
\end{lemma}
\begin{proof}
  Denoting
  $$
  \widetilde\lambda_{\pm} = \frac{1}{2}\left(-\left(c+\dfrac{1}{r_1}\right) \pm \sqrt{\left(c+\dfrac{1}{r_1}\right)^2+4\varepsilon}\right), \qquad \widetilde\lambda_- < 0 < \widetilde\lambda_+,
  $$
  we have
  $$
  \psi(r) = \frac{u_0}{\sqrt{(c+\frac{1}{2})^2+4\varepsilon}}\left(\widetilde\lambda_+ e^{\widetilde\lambda_- r} - \widetilde\lambda_- e^{\widetilde\lambda_+ r}\right).
  $$
  We verify easily that $\psi$ is a continuous, differentiable and increasing function on $\RR^+$. Moreover, $\psi(0) = u_0 < 1$, $\lim_{r\to +\infty} \psi(r) = +\infty$. Hence, there exists $L$ such that $\psi(L) = 1$.
  Furthermore, like for point $(ii)$ in Lemma \ref{lem:phi1}, we obtain by simple computations
  $$
  \psi'(r) < \widetilde\lambda_+ \psi(r) < \sqrt{\varepsilon} \psi(r).
  $$
  
\end{proof}

\begin{lemma}\label{lem:phi2}
  Under the same assumption as in Lemma \ref{lem:psi}, let us fix $r_2 = L+r_1$ and define
  \begin{equation}\label{def:phi2}
    \phi_2(x, t) = \psi(|x|-(r_1+ct)),
  \end{equation}
  where $\psi$ is defined in Lemma \ref{lem:psi}.
  Then, the function $\phi_2$ is a super-solution of the equation $\partial_t u - \Delta u = - \varepsilon u$ on $\Omega^2_t$ with Dirichlet boundary conditions $u = u_0$ on $\{|x|=r_1+ct\}$ and $u=1$  on $\{|x|=r_2+ct\}$.

  Moreover, on the set $\{|x|=r_1+ct\}$, we have $\partial_r \phi_2(x,t) = 0$, and on $\{|x|=r_2+ct\}$, we have $\partial_r \phi_2(x,t)>0$.
\end{lemma}
\begin{proof}
  Indeed we verify easily that, denoting $r=|x|$,
  $$
  \partial_t \phi_2 - \Delta \phi_2 + \varepsilon \phi_2 = -(c+\frac 1r) \psi' - \psi'' + \varepsilon \psi = (\frac{1}{r_1} - \frac{1}{r}) \psi'.
  $$
  This latter quantity is nonnegative since on $\Omega^2_t$ we have $r_1+ct < r < r_2 + ct$ and $\psi$ is increasing.
  Finally, the Dirichlet boundary conditions follows straightforwardly from the definition of $\psi$ in Lemma \ref{lem:psi}.
\end{proof}

\subsection{Construction of a super-solution}

We first recall the notation $(E^*,M^*,F^*)$ for the positive equilibrium of system \eqref{syst:E}--\eqref{syst:F}. With the notations of Lemma \ref{lem:phi1} and Lemma \ref{lem:phi2}, we define $\barF(x,t)$ on $\RR^2$ by
\begin{equation}\label{eq:barF}
  \barF(x,t) =\left\lbrace 
    \begin{aligned}
      &F^* \alpha(t) \beta(0)\quad &&  \text{ on } \Omega^0_t, \\
      &F^* \phi_1(x,t) && \text{ on } \Omega^1_t ,  \\
      &F^* \phi_2(x,t) && \text{ on } \Omega^2_t,  \\
      &F^* \quad && \text{ on } \Omega^3_t. 
\end{aligned}\right.
\end{equation}
Notice that by construction the function $\barF$ is radially symmetric and nondecreasing with respect to $|x|$.

\begin{lemma}\label{lem:Fsuper}
  Let $\mu>0$, $r_1>0$, $\varepsilon>0$, and $u_0\in (0,1)$, and let $0<\underline{c}<c'<c$ with $c'> \frac 23 c$. Let $L = r_2 - r_1$ large enough as in Lemma \ref{lem:psi}.
  We define $g(x,t)$ by
  $$
  g(x,t) = \frac{\mu}{4} \mathbf{1}_{\Omega^0_t} + \mu \mathbf{1}_{\Omega^1_t} + \varepsilon \mathbf{1}_{\Omega_t^2}.
  $$
  Then, $\barF$ defined in \eqref{eq:barF} is a super-solution in $\RR^2$ of the equation
  $$
  \partial_t v - \Delta v + g(x,t) v = 0, \quad v(t=0) \leq \barF(x,0).
  $$
\end{lemma}
\begin{proof}
  By construction, $\barF$ is continuous on $\RR^2$, and for all fixed $t>0$ we have $\partial_r \barF(x,t) \geq 0$.
  From the definition of $\phi_1$ and $\phi_2$ in Lemma \ref{lem:phi1} and Lemma \ref{lem:phi2}, noticing also that on the boundary $\{|x|=r_1+ct\}$ we have $\partial_r \phi_1(x,t) \geq \partial_r \phi_2(x,t) = 0$, it is clear that
  $$
  \partial_t \barF - \Delta \barF + g(x,t) \barF \geq 0, \quad \text{ on } \Omega_t^1 \cup \Omega_t^2.
  $$

  On $\Omega_t^3$, $\barF$ is a constant therefore it is a super-solution since $g$ is nonnegative and on the boundary $\{|x|=r_2+ct\}$ we have $\partial_r \phi_2(x,t) \geq 0 = \partial_r \barF(x,t)$.
  
  Finally, on $\Omega_t^0$, we have $\Delta \barF = 0$ and $\partial_t \barF > -\frac{\mu}{4} \barF$ (see Lemma \ref{lem:phi1} $(i)$). We conclude by noticing also that the derivatives coincide at the boundary $\{|x|=r_1+c't\}$.
\end{proof}

\begin{lemma}\label{lem:Esuper}
  Under the same assumptions as in Lemma \ref{lem:Fsuper}, let us define $\barE$ as the solution of the equation
  \begin{equation}\label{def:E}
  \partial_t \barE = b\barF \left(1-\frac{\barE}{K}\right) - (\mu_E+\nu_E) \barE,
  \qquad \barE(t=0) = E^0 \leq \min\{K,C_0 \barF(x,0)\}, 
  \end{equation}
  for some positive constant $C_0$, $\barF$ being defined in \eqref{eq:barF}.
  Then, there exists $C_1\geq C_0$ large enough and $0<\mu$, $0<\varepsilon$ small enough, such that for all $t>0$ and $x\in\RR^2$, $\barE(t,x) \leq C_1 \barF(t,x)$.
\end{lemma}
\begin{proof}
  We verify that there exist $C>0$ large enough and $\mu$ small enough, such that $C\barF$ is a super-solution of the equation for $\barE$, i.e.
  $$
  C\partial_t \barF - b\barF \left(1-\frac{C\barF}{K}\right) + (\mu_E+\nu_E) C\barF \geq 0.
  $$

  On $\Omega_t^0$, we compute
  \begin{align*}
    C\partial_t \barF - b\barF \left(1-\frac{C\barF}{K}\right) + (\mu_E+\nu_E) C\barF
    & = F^* \beta(0) \left( C\alpha'(t) + \alpha(t)\left(C(\mu_E+\nu_E) - b+\frac{C b F^* \alpha(t)\beta(0)}{K}\right)\right) \\
     & \geq F^* C\alpha(t) \beta(0) \left( -\frac{\mu}{4} + \mu_E+\nu_E - \frac{b}{C}+\frac{ b F^* \alpha(t)\beta(0)}{K}\right).
  \end{align*}
  Then, if $C$ is large enough and $\mu$ small enough, this latter term is nonnegative.

  On $\Omega_t^1$, we have
  \begin{align*}
    C\partial_t \barF - b\barF \left(1-\frac{C\barF}{K}\right) + (\mu_E+\nu_E) C\barF
    & = F^* \left( C\partial_t \phi_1 - b \phi_1 \left( 1 - \frac{C F^* \phi_1}{K}\right) + (\mu_E+\nu_E) C \phi_1\right) \\
    & = F^* \left( C \alpha' \beta - c'C \alpha \beta' - b \phi_1\left( 1 - \frac{C F^* \phi_1}{K}\right) + (\mu_E+\nu_E) C \phi_1\right) \\
    & > F^* \left( -C (\frac{\mu}{4} + c' \sqrt{\frac{\mu}{2}}) \phi_1- b \phi_1\left( 1 - \frac{C F^* \phi_1}{K}\right) + (\mu_E+\nu_E) C \phi_1\right),  
  \end{align*}
  where we use Lemma \ref{lem:phi1} (i) and (ii). We arrive at
  $$
  C\partial_t \barF - b\barF \left(1-\frac{C\barF}{K}\right) + (\mu_E+\nu_E) C\barF
  > C F^* \phi_1 \left( -(\frac{\mu}{4} + c' \sqrt{\frac{\mu}{2}}) - \frac{b}{C} + \mu_E + \nu_E \right).
  $$
  This latter term is nonnegative provided $C$ is large enough and $\mu$ is small enough.

  On $\Omega_t^2$, by the same token as above, we compute
  \begin{align*}
    C\partial_t \barF - b\barF \left(1-\frac{C\barF}{K}\right) + (\mu_E+\nu_E) C\barF
    & = C F^* \left(c\psi'-b\psi\left(\frac{1}{C}-\dfrac{F^*\psi}{K}\right) + (\mu_E+\nu_E) \psi\right) \\
    & > C F^* \psi \left(-c\sqrt{\varepsilon} - \frac{b}{C} + \mu_E+\nu_E\right).
  \end{align*}
  The latter term is nonnegative provided $\varepsilon$ is small enough and $C$ is large enough.

  Finally, on $\Omega_t^3$, we have $\barF = F^*$ is a constant and $\barE$ is bounded. Therefore, $\barE\leq C\barF$ on $\Omega_t^3$ for $C$ large enough.

\end{proof}

\begin{lemma}\label{lem:Msuper}
  Under the same assumptions as in Lemma \ref{lem:Fsuper}, let us define $\barM$ as the solution of the equation
  \begin{equation}\label{def:M}
  \partial_t \barM - \Delta \barM = (1-r)\nu_E \barE - \mu_M \barM,
  \qquad \barM(t=0) = M^0 \leq C_0 \barF(x,0),
  \end{equation}
  with $\barE$ defined in Lemma \ref{lem:Esuper}.
  Then, if $\mu>0$ and $\varepsilon>0$ are small enough, there exists $C_2\geq C_0$ large enough such that for all $t>0$ and $x\in\RR^2$, $\barM(x,t) \leq C_2 \barF(x,t)$.
\end{lemma}
\begin{proof}
  We compute for some constant $C>0$, using Lemma \ref{lem:Fsuper},
  \begin{align*}
    C\partial_t \barF - C\Delta \barF  + \mu_M C \barF - (1-r)\nu_E \barE
    & = -C g(x,t) \barF + C \mu_M \barF - (1-r)\nu_E \barE  \\
    & \geq (-C g(x,t) + C \mu_M - (1-r)\nu_E C_1) \barF,
  \end{align*}
  where we use Lemma \ref{lem:Esuper} for the last inequality.
  Hence, if we take $\mu$ and $\varepsilon$ small enough such that $\mu_M+g>0$, we may take $C$ large enough such that the right hand side of the latter inequality is nonnegative. It implies that for $C$ large enough $C\barF$ is a super-solution of equation \eqref{def:M} which allows to conclude the proof.
\end{proof}

\begin{proposition}\label{prop1}
  Let $\mu>0$, $r_1>0$, $\varepsilon>0$, $u_0\in (0,1)$, and $0<\underline{c}<c'<c$ with $c'> \frac 23 c$. Let $L = r_2 - r_1$ large enough as in Lemma \ref{lem:psi}.

  Let us assume that
  \begin{itemize}
  \item[(i)] In the \textit{bistable case $\Gamma(M) = 1-e^{-\gamma M}$}, we have
    \begin{equation}\label{hypMs1}
      M_s \geq \overline{M_s} \mathbf{1}_{\{r_1+ct\leq |x| \leq r_2+ct\}}.
    \end{equation}
  \item[(ii)] In the \textit{monostable case $\Gamma(M) = 1$}, we have
    \begin{equation}
      \label{hypMs2}
      M_s \geq \overline{M_s} \mathbf{1}_{\{r_1+ct\leq |x| \leq r_2+ct\}} + \overline{M_s} \frac{\barF(x,t)}{u_0} \mathbf{1}_{\{|x| \leq r_1+ct\}}.
    \end{equation}
  \end{itemize}
  Then, for $\mu$, $\varepsilon$ and $u_0$ small enough and $\overline{M_s}$ large enough, $(\barE,\barM,\barF)$ defined respectively in \eqref{def:E}, \eqref{eq:barF}, \eqref{def:M}, is a super solution of system \eqref{syst:E}--\eqref{syst:F} with initial data satisfying \eqref{hyp:init}.
  
\end{proposition}
\begin{proof}
  We first notice that due to assumption \eqref{hyp:init}, the conditions on the initial data are clearly satisfied.
  Moreover, from Lemma \ref{lem:Esuper} and Lemma \ref{lem:Msuper}, we already know that $\barE$ and $\barM$ are super-solutions. Then, we are left to prove that $\barF$ is a super-solution for \eqref{syst:F}. From Lemma \ref{lem:Fsuper} it is enough to prove that
  \begin{equation}\label{estimFbar}
    r\nu_E \barE \frac{\barM}{\barM+\gamma_s M_s} \Gamma(\barM+ \gamma_s M_s) - \mu_F \barF \leq -g(t,x)\barF,
  \end{equation}
  where we recall that $g$ is defined in the statement of Lemma \ref{lem:Fsuper}.
  
  On the set $\{r_1+ct < |x| < r_2+ct\}$, we have $\barM\leq C_2 \barF \leq C_2 F^*$, and
  \begin{align*}
    r\nu_E \barE \frac{\barM}{\barM+\gamma_s M_s} \Gamma(\barM+\gamma_s M_s) - \mu_F \barF
    & \leq r\nu_E \barE \frac{C_2 F^*}{C_2 F^*+\gamma_s \overline{M_s}} - \mu_F \barF  \\
    & \leq r\nu_E C_1 \barF \frac{C_2 F^*}{C_2 F^*+\gamma_s \overline{M_s}} - \mu_F \barF < -\varepsilon\barF,
  \end{align*}
  for $\overline{M_s}$ large enough.

  By definition, on the set $\{|x|< r_1+ct\}$, we have $\barF\leq u_0$, which implies, using  Lemma \ref{lem:Esuper} and Lemma \ref{lem:Msuper}, that $\barE \leq C_1 u_0$ and $\barM \leq C_2 u_0$.
  Then, recalling that $M_s \mapsto \frac{\barM}{\barM+\gamma_s M_s} \Gamma(\barM+\gamma_s M_s)$ is nonincreasing, we have
  $$
  \frac{\barM}{\barM+\gamma_s M_s} \Gamma(\barM+\gamma_s M_s) \leq \Gamma(\barM) \leq \Gamma(C_2 u_0).
  $$
  Therefore, in the bistable case $(i)$, we have
  \begin{align*}
    r\nu_E \barE \frac{\barM}{\barM+\gamma_s M_s} \Gamma(\barM+\gamma_s M_s)
    & \leq r\nu_E \barE \left(1-e^{-\gamma C_2 u_0}\right)   \\
    & \leq r\nu_E C_1 \left(1-e^{-\gamma C_2 u_0}\right) \barF,
  \end{align*}
  where we use Lemma \ref{lem:Esuper} for the last inequality. Hence, if we take $\mu<\mu_F$, there exists $u_0$ small enough such that
  $$
  r\nu_E \barE \frac{\barM}{\barM+\gamma_s M_s} \Gamma(\barM+\gamma_s M_s)
  \leq (\mu_F - \mu) \barF.
  $$
  This implies that on the set $\{|x|<r_1+ct\}$, in the bistable case $(i)$, we have
  $$
  r\nu_E \barE \frac{\barM}{\barM+\gamma_s M_s} \Gamma(\barM+\gamma_s M_s) - \mu_F \barF \leq -\mu \barF.
  $$
  In the monostable case $(ii)$, by assumption \eqref{hypMs2} on $M_s$, we have, on the set $\{|x|<r_1+ct\}$,
  \begin{align*}
    r\nu_E \barE \frac{\barM}{\barM+\gamma_s M_s} \Gamma(\barM+\gamma_s M_s)
    & \leq r\nu_E \barE \frac{C_2 \barF}{C_2 \barF + \gamma_s \overline{M_s} \frac{\barF}{u_0} }   \\
    & \leq r\nu_E C_1 \frac{C_2 u_0}{C_2 u_0 + \gamma_s\overline{M_s}} \barF,
  \end{align*}
  where we use the estimate $\barE\leq C_1 \barF$ from Lemma \ref{lem:Esuper} for the last inequality. Then, for $\overline{M_s}$ large enough, we have the desired estimate 
  $$
  r\nu_E \barE \frac{\barM}{\barM+\gamma_s M_s} \Gamma(\barM) - \mu_F \barF \leq -\mu \barF.
  $$
  
\end{proof}

The following lemma shows how to obtain conditions \eqref{hypMs1} and \eqref{hypMs2}.

\begin{lemma}\label{lem:Ms}
Let $c>0$, $0<R_1<r_1<r_2<R_2$ be fixed. 
Let $M_s$ be the solution of the equation
\begin{equation}\label{equaMs}
\partial_t M_s - \Delta M_s = \Lambda - \mu_s M_s, 
\qquad M_s(t=0) = M_s^0,
\end{equation}
with $M_s^0$ as in \eqref{hyp:init} with $R_0^0>R_2$.
Then, we have~:
\begin{itemize}
\item[(i)] If $\Lambda = \bar{\Lambda} \mathbf{1}_{\{R_1+ct<|x|<R_2+ct\}}$, then the solution of equation  \eqref{equaMs} verifies
$$
M_s(x,t) \geq C_{12} \bar{\Lambda} \mathbf{1}_{\{r_1+ct<|x|<r_2+ct\}},
$$
where $C_{12}$ is a constant depending on $c+\frac{1}{r_2}$, $r_1-R_1$, $R_2-r_2$, and $\mu_s$.
\item[(ii)] If $\Lambda = \bar{\Lambda} \mathbf{1}_{\{R_1+ct<|x|<R_2+ct\}} + \bar{\Lambda} e^{\eta([x|-R_1-ct)}\mathbf{1}_{\{|x|<R_1+ct\}}$, for some $\eta>0$, then the solution of \eqref{equaMs} satisfies
$$
M_s(x,t) \geq C_{12} \bar{\Lambda} \Big(\mathbf{1}_{\{r_1+ct<|x|<r_2+ct\}} + e^{\eta(|x|-r_1-ct)} \mathbf{1}_{\{|x|\leq r_1+ct\}}\Big),
$$
where $C_{12}$ is a constant depending on $c+\frac{1}{r_2}$, $r_1-R_1$, $R_2-r_2$, and $\mu_s$.
\end{itemize}
\end{lemma}

\begin{proof}
The proof relies on the construction of a sub-solution for equation \eqref{equaMs}.
First, it is clear that $M_s^0$ verifies the inequalities announced.
\begin{itemize}
 \item[(i)] Let us introduce the function $m$ defined on $\RR$ by
$$
m(r) = \widehat{\mathrm{M}} 
\begin{cases}
e^{-a(r-r_1)^2}, & \qquad \text{ on } (-\infty,r_1),  \\
1, & \qquad \text{ on } (r_1,r_2), \\
e^{-b(r-r_2)^2}, & \qquad \text{ on } (r_2,+\infty),
\end{cases}
$$
for some constant $\widehat{\mathrm{M}}$ which will be fixed later.
Then, for $t>0$, $x\in\RR^2$, we define the function $m_s(x,t) = m(|x|-ct)$. Clearly, 
$m_s\geq \widehat{\mathrm{M}} \mathbf{1}_{\{r_1+ct<|x|<r_2+ct\}}$. 
We compute
\begin{align*}
&\partial_t m_s - \Delta m_s + \mu_s m_s 
 = - m'' - \left(c+\dfrac{1}{|x|}\right) m' + \mu_s m  \\
& = \widehat{\mathrm{M}}
\begin{cases}
    e^{-a(|x|-r_1-ct)^2} \left(\mu_s+2a(|x|-r_1-ct) \Big(c+\dfrac{1}{|x|}\Big)+2a-4a^2(|x|-r_1-ct)^2\right),  \text{ if } |x|<r_1+ct, \\[2mm]
    \mu_s, \hspace{10.5cm}  \text{ if } |x|\in (r_1+ct,r_2+ct), \\[2mm]
    e^{-b(|x|-r_2-ct)^2} \left(\mu_s+2b(|x|-r_2-ct) \Big(c+\dfrac{1}{|x|}\Big)+2b-4b^2(|x|-r_2-ct)^2\right),  \text{ if } |x|>r_2+ct. \\
\end{cases}
\end{align*}
For $|x|<r_1+ct$, we compute
$$
\mu_s+2a(|x|-r_1-ct) \Big(c+\dfrac{1}{r}\Big)+2a-4a^2(|x|-r_1-ct)^2 \leq 
\mu_s+2a-4a^2(|x|-r_1-ct)^2.
$$
In particular, for all $|x|\leq R_1+ct$, we have
$$
\mu_s+2a(|x|-r_1-ct) \Big(c+\dfrac{1}{|x|}\Big)+2a-4a^2(|x|-r_1-ct)^2 \leq 
\mu_s + 2a - 4a^2(R_1-r_1)^2.
$$
This right hand side is non-positive if $a \geq \dfrac{1+\sqrt{1+4(r_1-R_1)^2 \mu_s}}{4(r_1-R_1)^2}$. Moreover, if $\widehat{\mathrm{M}}\leq \dfrac{\bar{\Lambda}}{2a+\mu_s}$, we obtain the estimate, for all $|x|\leq r_1+ct$,
\begin{equation}
    \label{eqM1}
\widehat{\mathrm{M}}\left(\mu_s+2a(|x|-r_1-ct) \Big(c+\dfrac{1}{|x|}\Big)+2a-4a^2(|x|-r_1-ct)^2\right) \leq \bar{\Lambda} \mathbf{1}_{\{R_1+ct<|x|<r_1+ct\}}.
\end{equation}
For $|x|>r_2+ct$, we have, for all $t>0$,
\begin{align*}
&\mu_s+2b(|x|-r_2-ct) \Big(c+\dfrac{1}{|x|}\Big)+2b-4b^2(|x|-r_2-ct)^2   \\
& \qquad \leq \mu_s+2b(|x|-r_2-ct) \Big(c+\dfrac{1}{r_2}\Big)+2b-4b^2(|x|-r_2-ct)^2 = \mathrm{P}(|x|-r_2-ct),
\end{align*}
where $\mathrm{P}(X) = \mu_s + 2b + 2b\Big(c+\dfrac{1}{r_2}\Big) X - 4 b^2 X^2$. 
This polynomial is maximum for $X=\frac{1}{4b}(c+\frac{1}{r_2})$ with maximum value given by $2b+\mu_s+\frac{1}{4}(c+\frac{1}{r_2})^2$.
Hence, if 
\begin{equation}\label{condM2}
\widehat{\mathrm{M}} \leq \frac{\bar{\Lambda}}{2b+\mu_s+\frac 14 (c+\frac{1}{r_2})^2},
\end{equation} 
we deduce that $\widehat{\mathrm{M}} \mathrm{P}(|x|-r_2-ct)\leq \bar{\Lambda}$.

Then, for all $X>R_2-r_2$, we have $P(X)<P(R_2-r_2)$ provided $R_2-r_2>\frac{1}{4b}(c+\frac{1}{r_2})$, which is equivalent to $b>\frac{1}{4(R_2-r_2)}(c+\frac{1}{r_2})$. 
Moreover, we verify easily that $\mathrm{P}(R_2-r_2)\leq 0$ for any 
\begin{equation}\label{condb}
b\geq \frac{1}{4(R_2-r_2)^2}\left((c+\frac{1}{r_2})(R_2-r_2)+1+\sqrt{(1+(c+\frac{1}{r_2})(R_2-r_2))^2+4\mu_s(R_2-r_2)^2}\right).
\end{equation}
As a consequence, we have proved that when $b$ and $\widehat{\mathrm{M}}$ verify respectively \eqref{condb} and \eqref{condM2}, then for all $|x|\geq r_2+ct$,
\begin{equation}
    \label{eqM2}
\widehat{\mathrm{M}}\left(\mu_s+2b(|x|-r_2-ct) \Big(c+\dfrac{1}{|x|}\Big)+2b-4b^2(|x|-r_2-ct)^2\right) \leq \bar{\Lambda} \mathbf{1}_{\{r_2+ct<|x|<R_2+ct\}}.
\end{equation}
Combining \eqref{eqM1} and \eqref{eqM2} we see that 
$$
\partial_t m_s - \Delta m_s + \mu_s m_s \leq \Lambda.
$$
Hence, $m_s$ is a sub-solution for equation \eqref{equaMs}, which implies $M_s\geq m_s\geq \widehat{\mathrm{M}} \mathbf{1}_{\{r_1+ct<|x|<r_2+ct\}}$.
We conclude the proof of this first point by taking 
$$
\widehat{\mathrm{M}} = \bar{\Lambda}\min\left(\frac{1}{2b+\mu_s+\frac 14 (c+\frac{1}{r_2})^2},\frac{1}{2a+\mu_s}\right),
$$
with $a$ and $b$ chosen as above.

\item[(ii)] 
We proceed in the same way for the proof of the second point. We first fix $\varepsilon>0$ such that $\eta(r_1-R_1) = (1+\varepsilon)\ln(1+\varepsilon)$ and we define $a_\varepsilon = \frac{\eta}{2(1+\varepsilon)(r_1-R_1)}$.
Then, we introduce the function 
$$
m(r) = \widehat{\mathrm{M}} 
\begin{cases}
e^{\eta(r-R_1)}, & \qquad \text{ on } (-\infty,R_1),  \\
(1+\varepsilon) e^{-a_\varepsilon (r-r_1)^2} & \qquad \text{ on } (R_1,r_1),  \\
1+ \varepsilon, & \qquad \text{ on } (r_1,r_2), \\
(1+\varepsilon) e^{-b(r-r_2)^2}, & \qquad \text{ on } (r_2,+\infty),
\end{cases}
$$
for some constant $\widehat{\mathrm{M}}$ which will be fixed later.
With this choice of $\varepsilon$ and $a_\varepsilon$, we have $m\in C^1(\RR)$.
As above, we define $m_s(x,t) = m(|x|-ct)$ for $t>0$ and $x\in \RR^2$ and we notice that 
\begin{equation}
    \label{estim:ms2}
m_s(x,t)\geq \widehat{\mathrm{M}} \left(\mathbf{1}_{\{r_1+ct<|x|<r_2+ct\}} + e^{\eta(|x|-r_1-ct} \mathbf{1}_{\{|x|\leq r_1+ct\}} \right).
\end{equation}
We show that we may find constants $b$, and $\widehat{\mathrm{M}}$ such that $m_s$ is a sub-solution of \eqref{equaMs}.

For $|x|<R_1+ct$, we have 
\begin{align*}
\partial_t m_s - \Delta m_s + \mu_s m_s 
& = - m'' - \left(c+\dfrac{1}{|x|}\right) m' + \mu_s m  \\
& = \widehat{\mathrm{M}} e^{\eta(|x|-R_1-ct)} \left(\mu_s - \Big(c+\dfrac{1}{|x|}\Big) \eta - \eta ^2\right)  \\
& \leq \widehat{\mathrm{M}} \mu_s e^{\eta(|x|-R_1-ct)}.
\end{align*}

For $R_1+ct<|x|<r_1+ct$, we obtain
\begin{align*}
&\partial_t m_s - \Delta m_s + \mu_s m_s 
 = - m'' - \left(c+\dfrac{1}{|x|}\right) m' + \mu_s m  \\
& = \widehat{\mathrm{M}} (1+\varepsilon) e^{-a_\varepsilon(|x|-r_1-ct)^2} \left(\mu_s+2a_\varepsilon(|x|-r_1-ct) \Big(c+\dfrac{1}{|x|}\Big)+2a_\varepsilon-4a_\varepsilon^2(|x|-r_1-ct)^2\right)  \\
& \leq \widehat{\mathrm{M}} (1+\varepsilon) e^{-a_\varepsilon(|x|-r_1-ct)^2} \left(\mu_s+2a_\varepsilon\right)
\end{align*}

For $r_1+ct<|x|<r_2+ct$, we have
\begin{align*}
\partial_t m_s - \Delta m_s + \mu_s m_s  = \widehat{\mathrm{M}} (1+\varepsilon) \mu_s.
\end{align*}

We treat the domain $|x|>r_2+ct$ as in point (i). 

Finally, by taking $b$ verifying \eqref{condb} and 
$$
\widehat{\mathrm{M}} = \frac{\bar{\Lambda}}{1+\varepsilon} \min\left(\frac{1}{2b+\mu_s+\frac 14 (c+\frac{1}{r_2})^2},\frac{1}{2a_\varepsilon+\mu_s}\right),
$$
we deduce that 
$$
\partial_t m_s - \Delta m_s + \mu_s m_s  \leq \Lambda.
$$
Hence $m_s$ is a subsolution and we conclude thanks to estimate \eqref{estim:ms2}
\end{itemize}

\end{proof}

\section{Conclusion of the proof}\label{sec:final}
\subsection{Proof of Theorem \ref{MainTheorem}}

To summarize, under the assumptions of Theorem \ref{MainTheorem}, we have constructed a super-solution $(\barE,\barM,\barF,M_s)$ (see Lemma \ref{lem:Ms} and Proposition \ref{prop1}) and a sub-solution $(\underline{\mathbf{E}},\underline{\mathbf{M}},\underline{\mathbf{F}},\underline{\mathbf{M_s}})$ of system \eqref{syst} (see Proposition \ref{prop:subsolsyst}.
Thanks to the comparison principle (see Lemma \ref{lem:compar}), we have~: 
$$
\forall\, (t,x)\in \RR_+\times\RR^2, \quad (\underline{\mathbf{E}},\underline{\mathbf{M}},\underline{\mathbf{F}}) (t,x) \leq (E,M,F)(t,x) \leq (\barE,\barM,\barF)(t,x).
$$
Moreover, by construction, we have, for any $\underline{c}\leq c'\leq c$, on $\Omega_t^0=B_{r_1+c't}$,
$$
\|(\barE,\barM,\barF)(t,x)\| \leq C \alpha(t), 
$$
for some constant $C>0$ and with $\alpha$ decreasing towards $0$ (see Lemma \ref{lem:phi1}). 
This allows to conclude the proof of point (i) of the Theorem.

For the second point, we have from Proposition \ref{prop:subsolsyst} that $(\underline{\mathbf{E}},\underline{\mathbf{M}},\underline{\mathbf{F}})$ is a function in translation at constant speed $c>0$ and it verifies 
$$
\lim_{|x|\to +\infty} (\underline{\mathbf{E}},\underline{\mathbf{M}},\underline{\mathbf{F}})(t,x) = (E^*,M^*,F^*).
$$
This yields point (ii) of the Theorem.

\subsection{Proof of Corollary \ref{cor:hetero}}

Let us denote $(E_1,M_1,F_1)$ the solution of system \eqref{syst:E}--\eqref{syst:F} with $K=K_1$, and $(E_2,M_2,F_2)$ the solution of system \eqref{syst:E}--\eqref{syst:F} with $K=K_2$.
By the comparison principle, since $K_1\leq K(x) \leq K_2$, we deduce that on $\RR_+\times \RR^2$,
$$
(E_1,M_1,F_1)(t,x) \leq (E,M,F)(t,x) \leq (E_2,M_2,F_2)(t,x).
$$
Then, by applying Theorem \ref{MainTheorem} for $(E_1,M_1,F_1)$ and $(E_2,M_2,F_2)$ we obtained the desired result.

\subsection{Numerical illustrations}
We carry out two-dimensional numerical simulations to illustrate Theorem \ref{MainTheorem}. We consider a case without releases of sterile males (Figure \ref{fig:2d:withoutst})  and a case with releases \ref{fig:2d:withst}. In both cases, we use a finite element method implemented in the FreeFem software (see \cite{Freefem}). We discretize a ball of radius $R=45$ km by $210050$ elements. The parameters are the ones given in Table \ref{tab:parametre} with $\gamma=0.5$. The simulations are initialized with a local population distribution satisfying \eqref{hyp:init1}, which includes a central disk free of mosquitoes.

\begin{figure}[h!]
    \centering
    \begin{subfigure}[t]{0.19\textwidth}
        \centering
        \includegraphics[width=\linewidth]{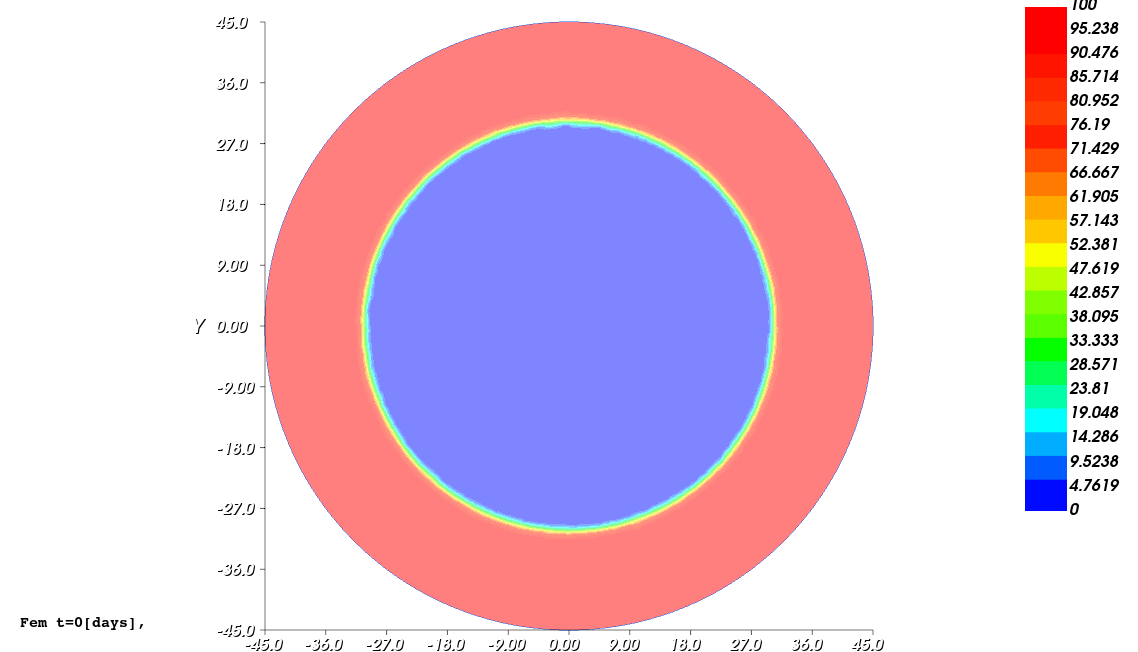}
        \caption{T = 0}
    \end{subfigure}
    \hfill
    \begin{subfigure}[t]{0.19\textwidth}
        \centering
        \includegraphics[width=\linewidth]{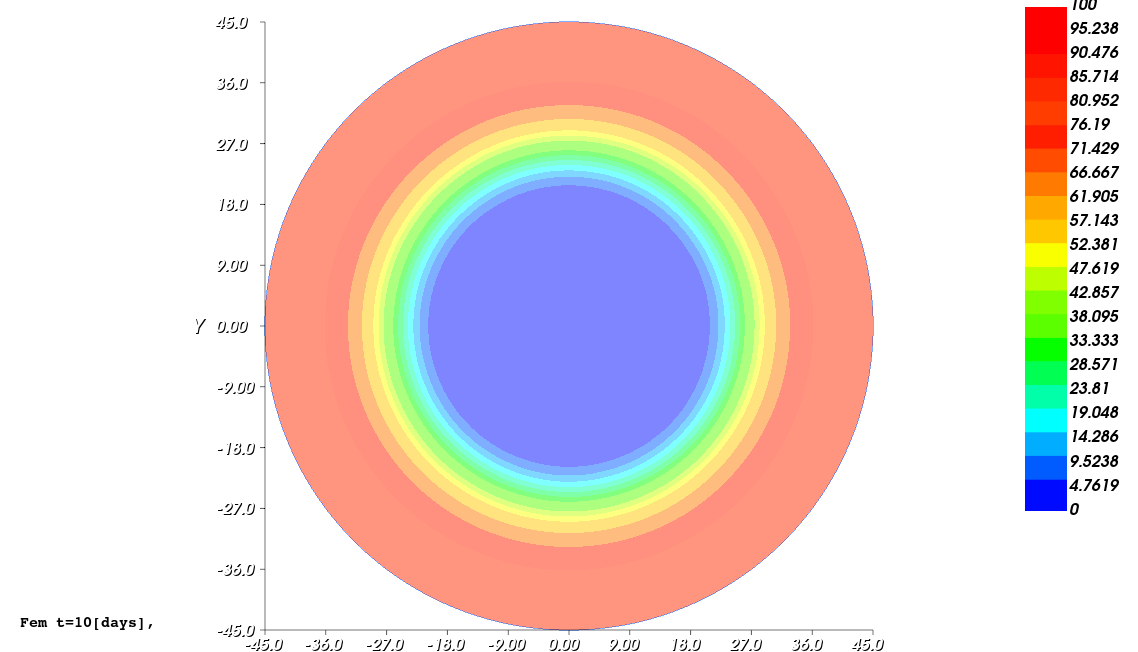}
        \caption{T = 10}
    \end{subfigure}
    \hfill
    \begin{subfigure}[t]{0.19\textwidth}
        \centering
        \includegraphics[width=\linewidth]{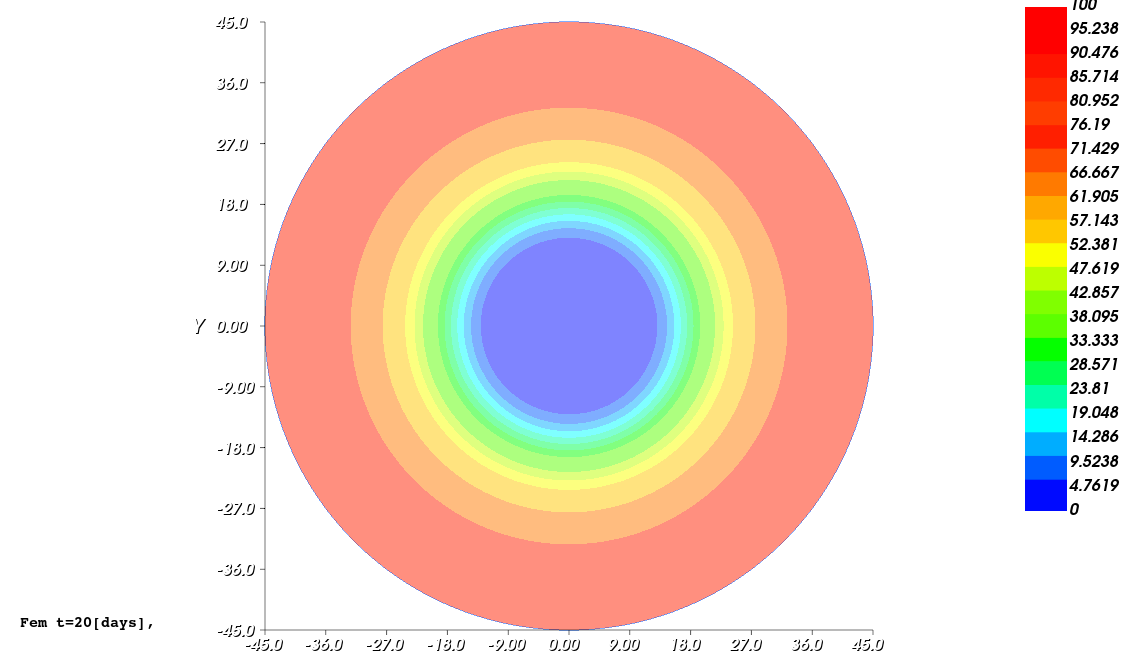}
        \caption{T = 20}
    \end{subfigure}
        \hfill
    \begin{subfigure}[t]{0.19\textwidth}
        \centering
        \includegraphics[width=\linewidth]{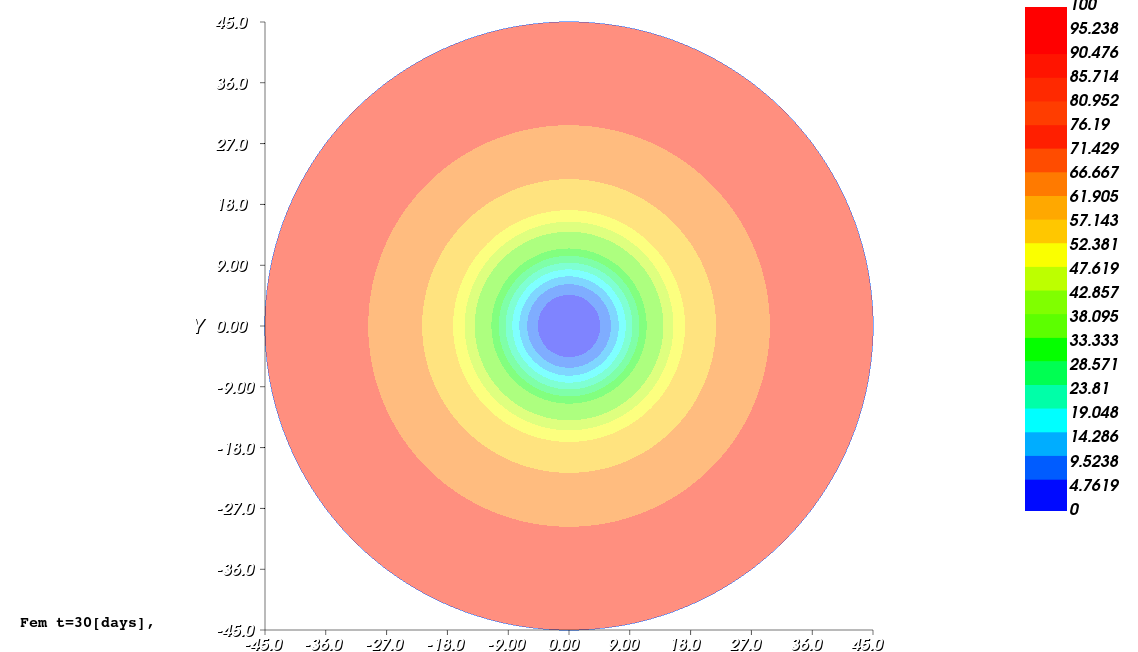}
        \caption{T = 30}
    \end{subfigure}
    \hfill
    \begin{subfigure}[t]{0.19\textwidth}
        \centering
        \includegraphics[width=\linewidth]{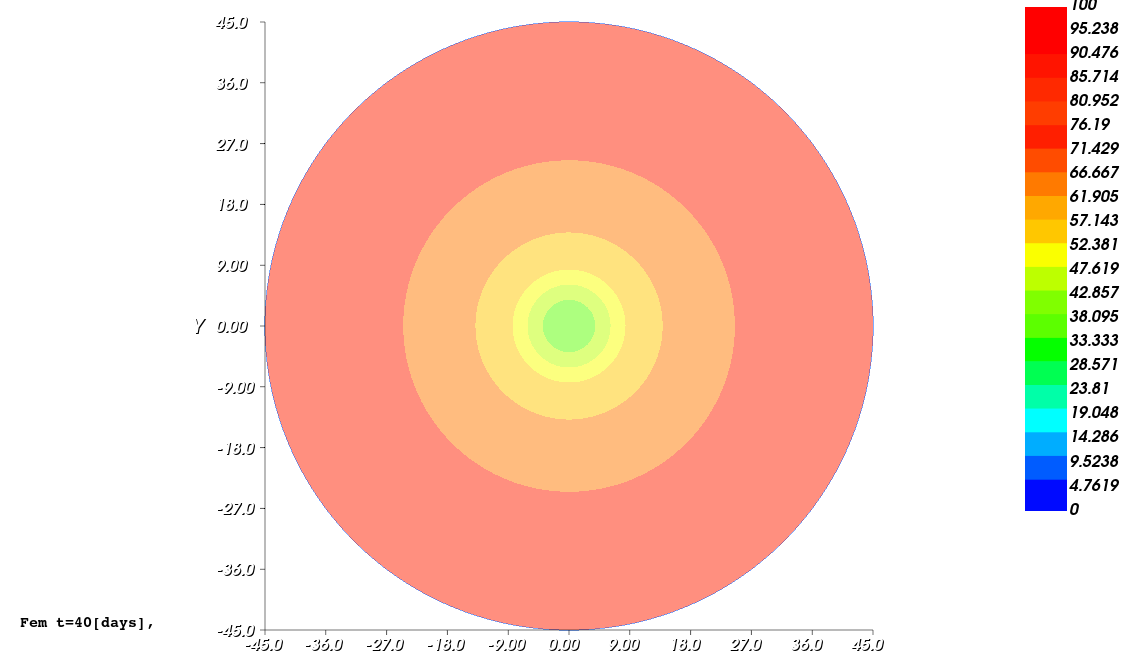}
        \caption{T = 40}
    \end{subfigure}
    \caption{Spatial distribution represented at different times T of female mosquitoes F, solution of \eqref{syst0} in a 2D homogeneous space without any releases of sterile males. Without control, we observe that mosquitoes are invading the central area.}
    \label{fig:2d:withoutst}
\end{figure}

With this choice of parameters, we have seen in Section \ref{sec:num1d} that without any intervention there is a natural invasion by the mosquito population. As expected, Figure~\ref{fig:2d:withoutst} shows that, in the absence of sterile males, the solution of the Cauchy problem \eqref{syst0} with initial data satisfying \eqref{hyp:init1} leads to the invasion of the central region by mosquitoes.

\begin{figure}[h!]
    \centering

    % Ligne 1
    \begin{subfigure}[t]{0.19\textwidth}
        \centering
        \includegraphics[width=\linewidth]{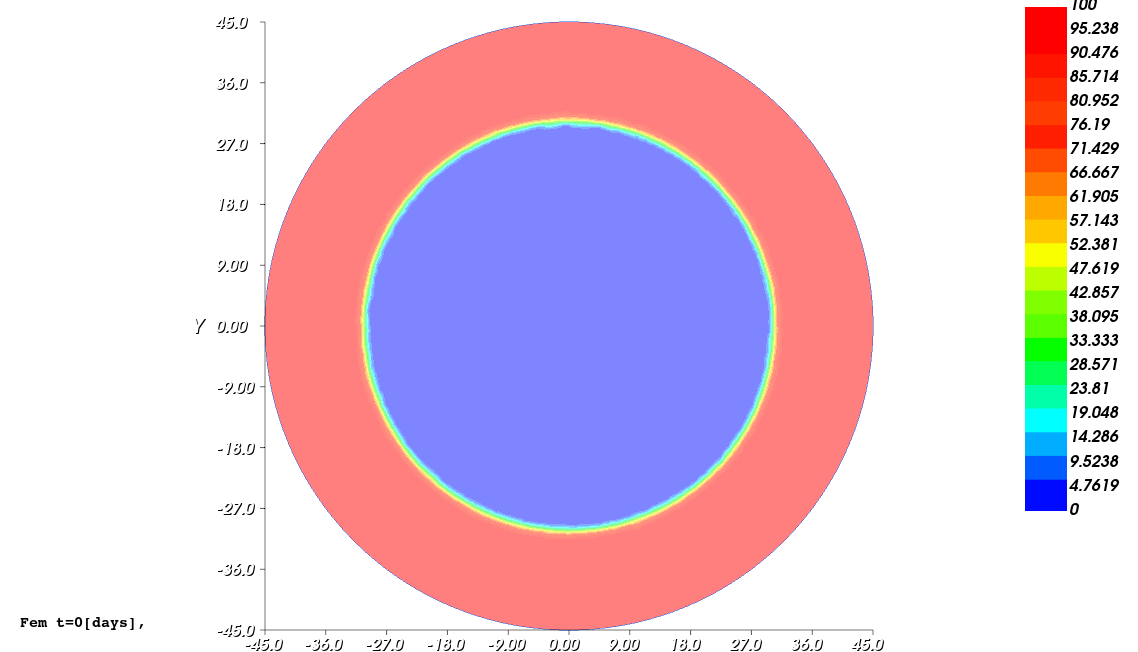}
    \end{subfigure}
    \hfill
    \begin{subfigure}[t]{0.19\textwidth}
        \centering
        \includegraphics[width=\linewidth]{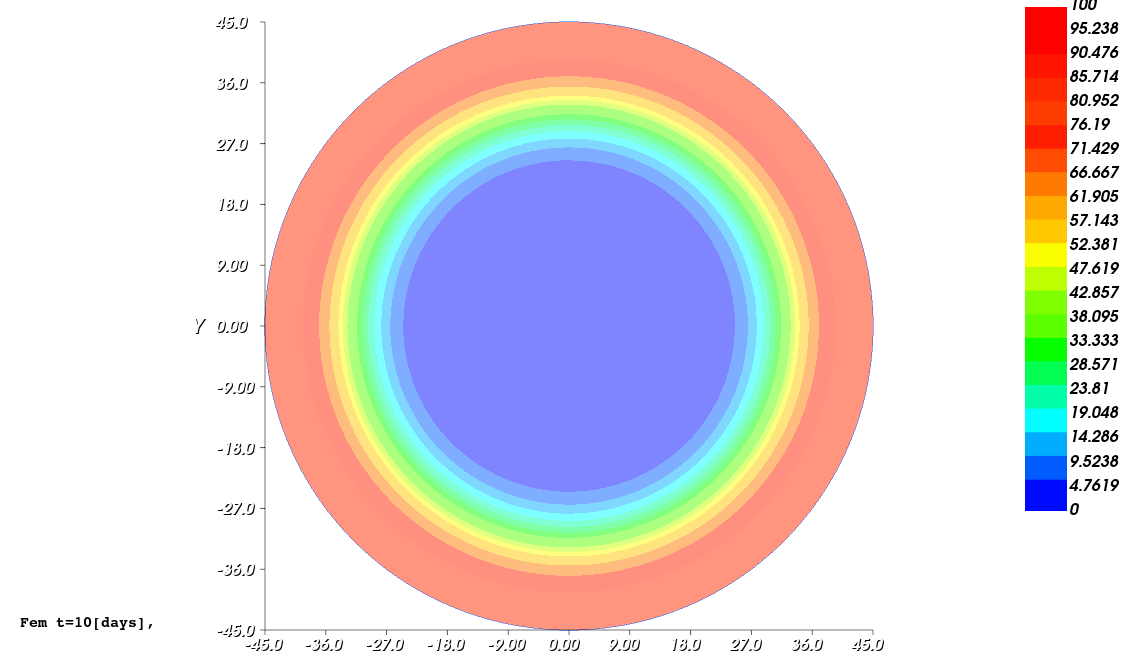}
    \end{subfigure}
    \hfill
    \begin{subfigure}[t]{0.19\textwidth}
        \centering
        \includegraphics[width=\linewidth]{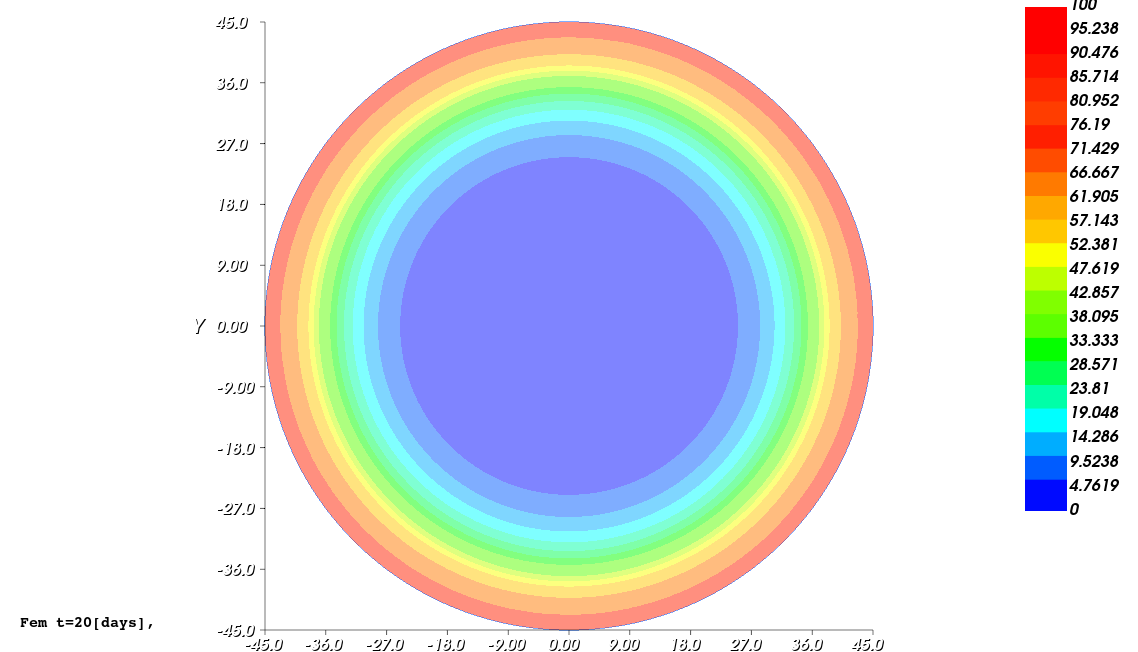}
    \end{subfigure}
    \hfill
    \begin{subfigure}[t]{0.19\textwidth}
        \centering
        \includegraphics[width=\linewidth]{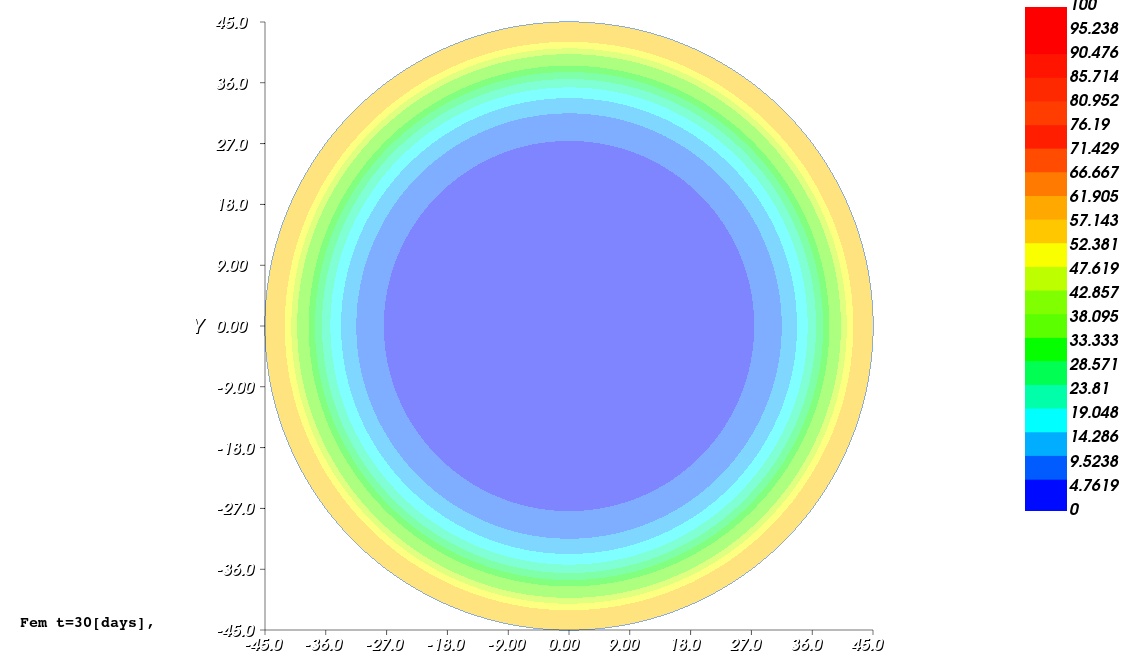}
    \end{subfigure}
           \hfill
    \begin{subfigure}[t]{0.19\textwidth}
        \centering
        \includegraphics[width=\linewidth]{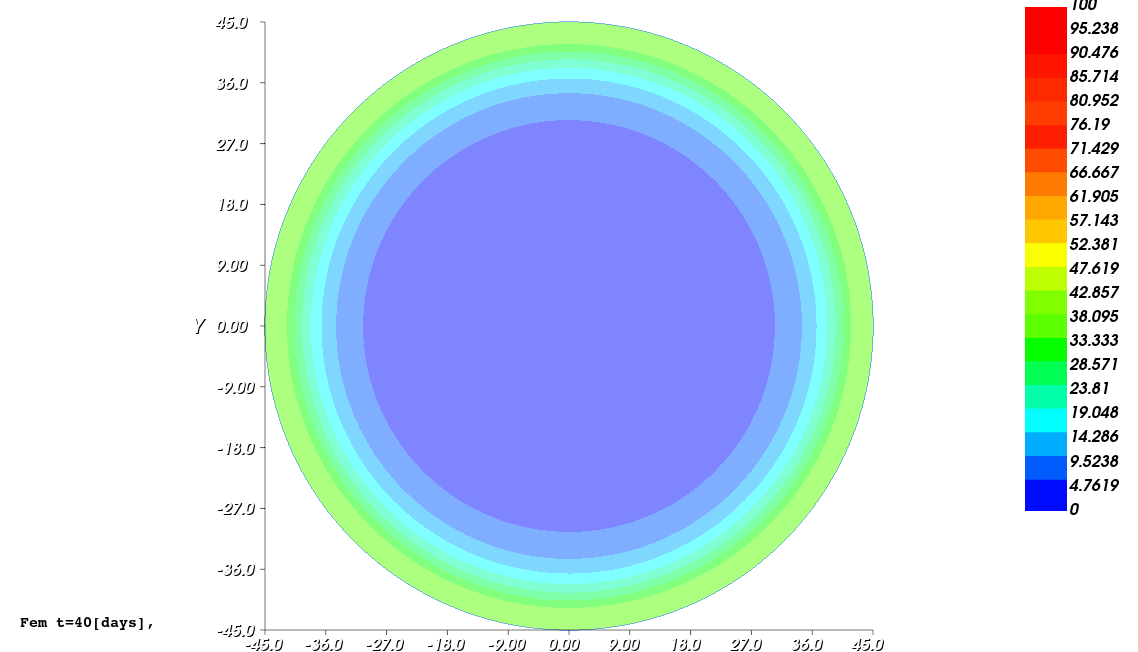}
    \end{subfigure}

    \vspace{1em} % Espace entre les deux lignes

    % Ligne 2
    \begin{subfigure}[t]{0.19\textwidth}
        \centering
        \includegraphics[width=\linewidth]{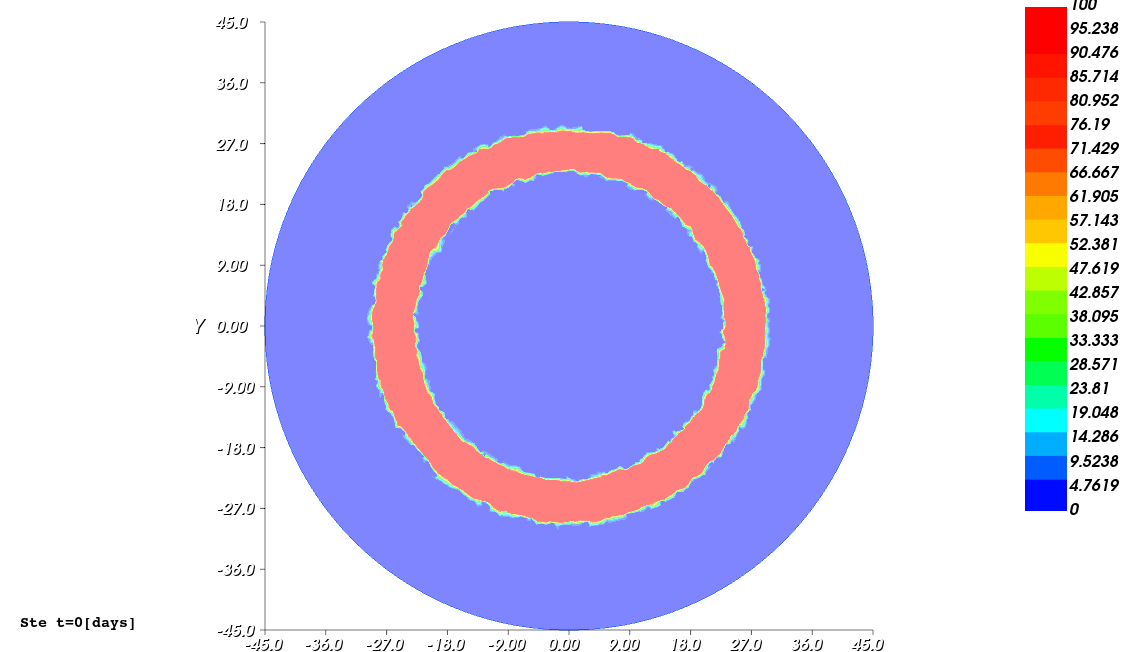}
        \caption{T = 0}
    \end{subfigure}
    \hfill
    \begin{subfigure}[t]{0.19\textwidth}
        \centering
        \includegraphics[width=\linewidth]{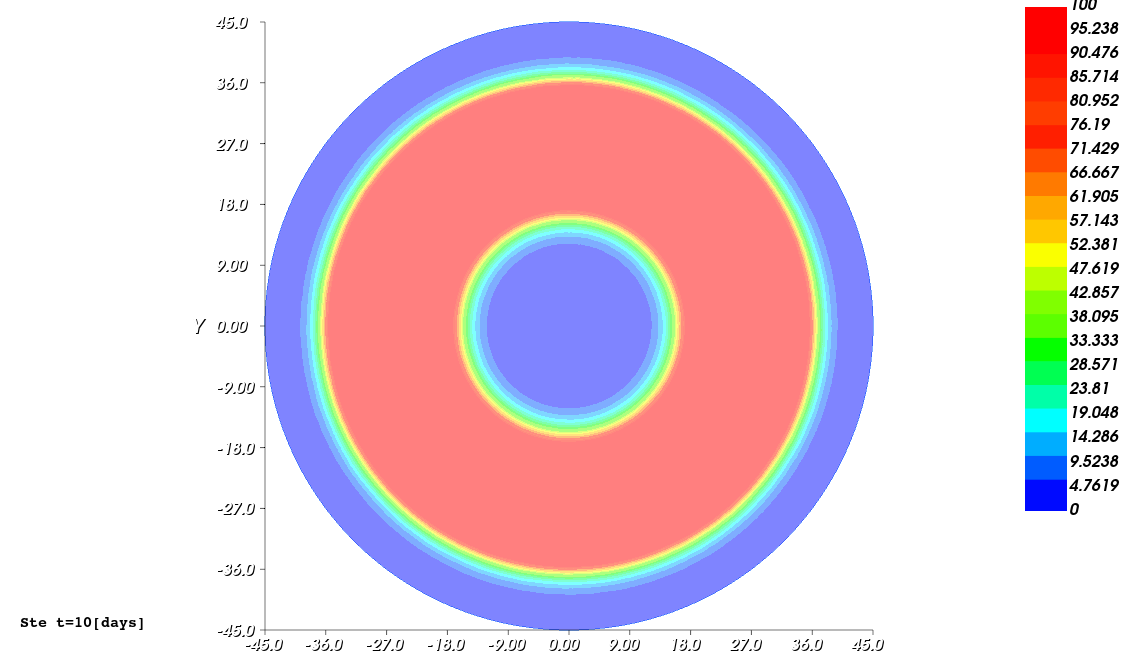}
        \caption{T = 10}
    \end{subfigure}
    \hfill
    \begin{subfigure}[t]{0.19\textwidth}
        \centering
        \includegraphics[width=\linewidth]{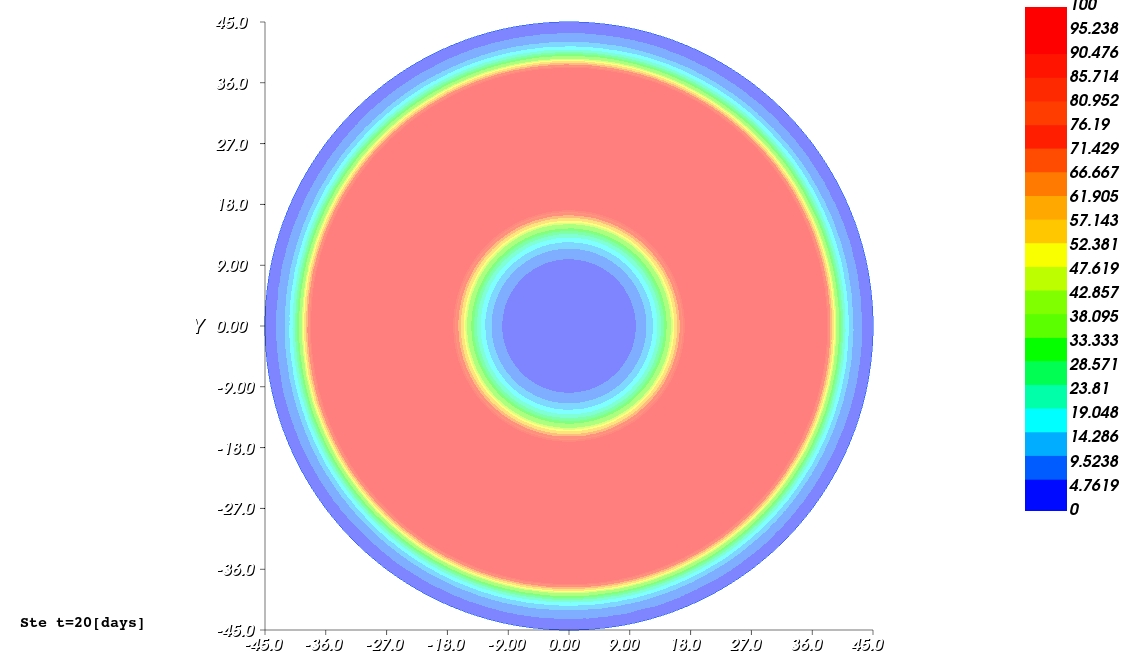}
        \caption{T = 20}
    \end{subfigure}
    \hfill
    \begin{subfigure}[t]{0.19\textwidth}
        \centering
        \includegraphics[width=\linewidth]{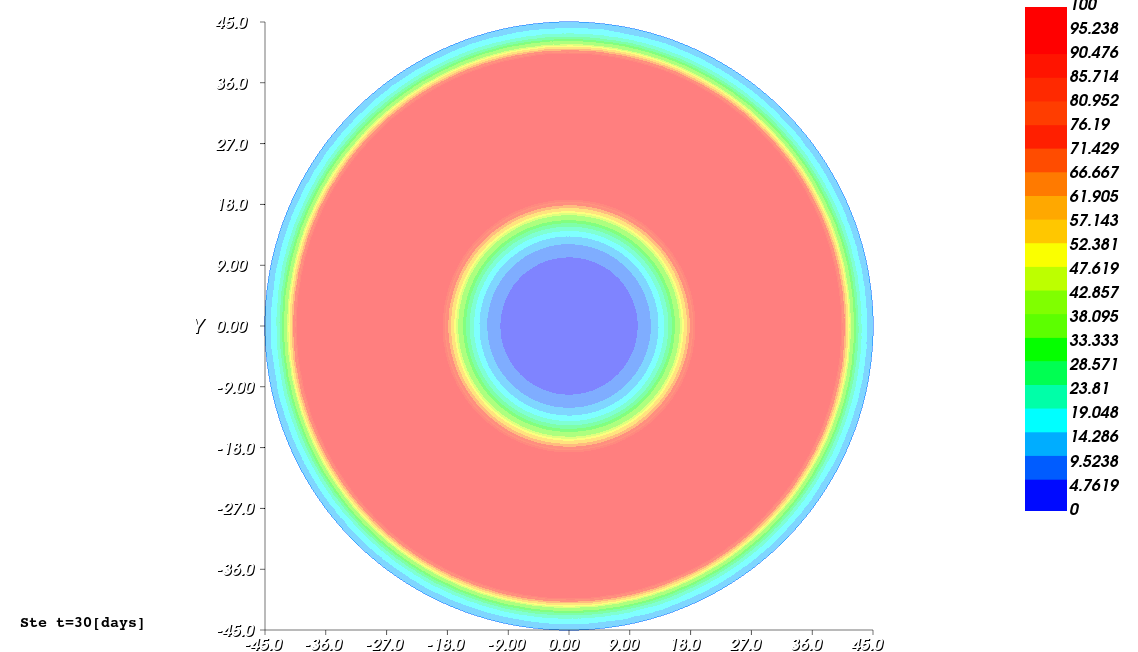}
        \caption{T = 30}
    \end{subfigure}
        \hfill
    \begin{subfigure}[t]{0.19\textwidth}
        \centering
        \includegraphics[width=\linewidth]{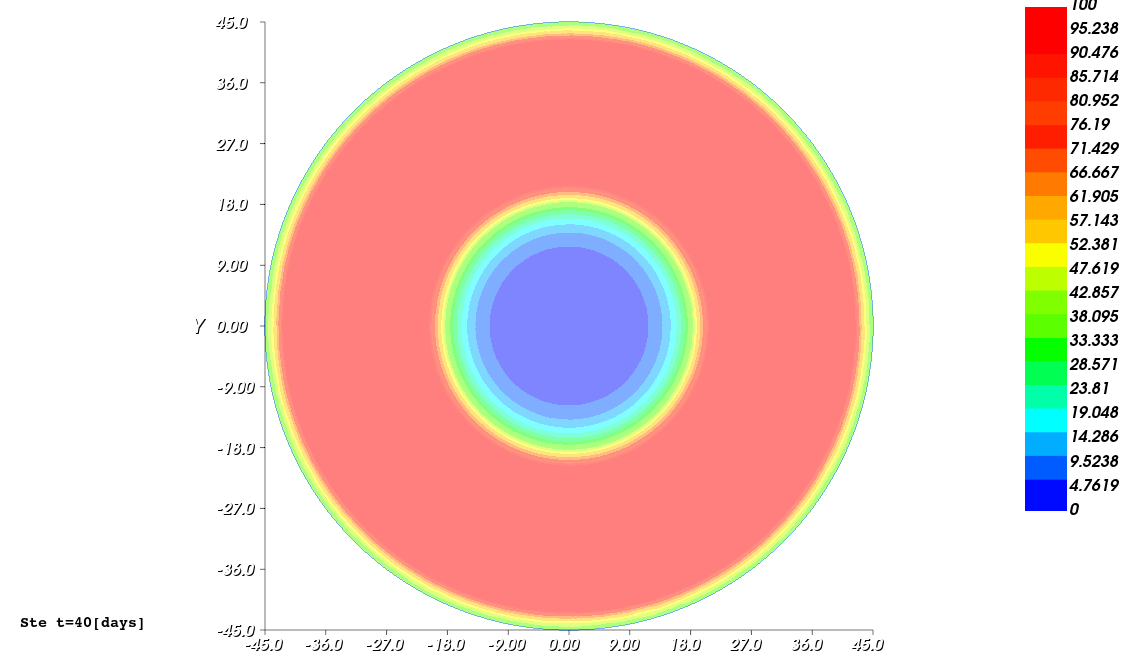}
        \caption{T = 40}
    \end{subfigure}

    \caption{Spatial distribution represented at different time T of female mosquitoes F (first line) and sterile males (second line) solutions of \eqref{syst0} in a 2D homogeneous space. With control, we observe extinction of the species in an expanding region.}
    \label{fig:2d:withst}    
\end{figure}

Then, Figure~\ref{fig:2d:withst} shows that the solution of the Cauchy problem \eqref{syst0}, with initial data satisfying \eqref{hyp:init1} and \eqref{hyp:init2}, and with a release of sterile males over time as in the statement of Theorem \ref{MainTheorem}, leads to a progressive decrease of the female population density to zero in an expanding region. In other words, the release of sterile mosquitoes allows us to enlarge the initial mosquito-free region, illustrating the success of the rolling carpet strategy in a two-dimensional domain.

\section*{Acknowledgments}
N.N. and N.V. acknowledge partial support from the STIC AmSud project BIO-CIVIP 23-STIC-02. A.L. acknowledges partial support from the ANR project “ReaCh” (ANR-23-CE40-0023-01). 
\noindent \begin{minipage}{0.10\textwidth}
	        \centering
            \includegraphics[width = \textwidth]{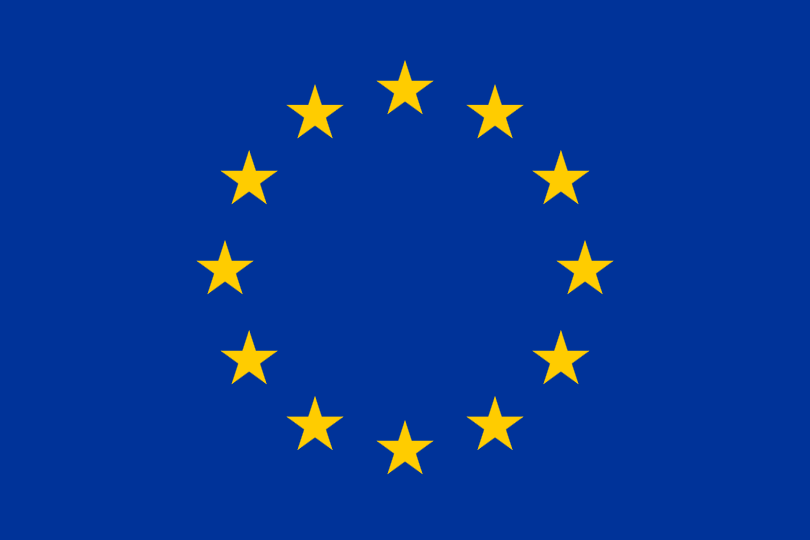} 
        \end{minipage}
        \hfill
        \begin{minipage}{0.85\textwidth}
            N.N. has received funding from the European Union's Horizon 2020 research and innovation program under the Marie Sklodowska-Curie grant agreement No 945332.
        \end{minipage}

\newpage
\appendix

\section{Appendix : Steady states in the bistable case}

This part is devoted to the proof of Lemma \ref{lem:stat} $(ii)$.
When $M_s=0$, we verify easily that the stationary solution $(E^*,M^*,F^*)$ satisfies
\begin{equation}\label{rel_stat1}
M^* = \frac{(1-\rho)\nu_E }{\mu_M} E^*, \qquad  E^*=\frac{b F^*}{\frac{b F^*}{K}+\mu_E+\nu_E}.
\end{equation}
We recall the definition of $\phi_0$:
$$
\phi_0(F) = \frac{(1-\rho) \nu_E b F}{\mu_M\frac{b F}{K}+\mu_M(\mu_E+\nu_E)}.
$$
Injecting into the stationary equation for $M$, we get $F^*=0$ or
\begin{equation}\label{rel_stat2}
\Gamma(\phi_0(F^*)) = \frac{\mu_F F^*}{\rho \nu_E K}+ \frac{1}{\mathcal{N}},
\end{equation}
where we recall that the basic offspring number $\mathcal{N}$ has been defined in \eqref{def:Nzeta}.

The condition $F^*=0$ gives the extinction equilibrium. The other equilibria (if they exist) are obtained by solving equation \eqref{rel_stat2}.
In order to solve this equation, we first notice that $\phi_0$ and $\Gamma$ are two increasing, continuous and concave functions on $\RR^+$. Hence, $F\mapsto \Gamma(\phi_0(F))$ is increasing, continuous, concave and bounded, whereas the right hand side of \eqref{rel_stat2} is affine~: we look for the intersection of an affine function with a concave function. Adding the fact that 
$\Gamma(\phi_0(0))=0 < \frac{1}{\mathcal{N}}$, we deduce that equation \eqref{rel_stat2} admits 0, 1 or 2 solutions (see Figure \ref{fig:Gamma}).
Moreover, in the case it has 2 solutions, denoted $F_1^*<F^*$, we have
\begin{equation}\label{rel_stat3}
\Gamma(\phi_0(F)) < \frac{\mu_F F}{\rho \nu_E K}+ \frac{1}{\mathcal{N}}, \text{ on } (0,F_1^*); \quad
\Gamma(\phi_0(F)) > \frac{\mu_F F}{\rho \nu_E K}+ \frac{1}{\mathcal{N}}, \text{ on } (F_1^*,F^*).
\end{equation}

\begin{figure}[ht]
\begin{center}
\begin{tikzpicture}[scale=1]
    %\draw[very thin,color=gray] (-0.1,-1.1) grid (3.9,3.9);
    \draw[->] (-0.1,0) -- (5,0) node[right] {$F$};
    \draw[->] (0,-0.1) -- (0,4.5);
    \draw[dashed,color=blue,domain=0:3.5] plot (\x,{\x+0.75});
    \draw[color=blue,domain=0:4.5]   plot (\x,{3.5*(1-exp(-0.5*\x))});   
    \draw[color=blue] (4,2.4) node {$\Gamma(\phi_0(F))$};
    \draw (0,-0.1) node[below left] {$0$};
\end{tikzpicture} 
\qquad
\begin{tikzpicture}[scale=1]
    %\draw[very thin,color=gray] (-0.1,-1.1) grid (3.9,3.9);
    \draw[->] (-0.1,0) -- (5,0) node[right] {$F$};
    \draw[->] (0,-0.1) -- (0,4.5);
    \draw[dashed,color=blue,domain=0:3.5] plot (\x,{\x+0.75});
    \draw[color=blue,domain=0:4.5]   plot (\x,{3.5*(1-exp(-\x))});   
    \draw[color=blue] (4,2.8) node {$\Gamma(\phi_0(F))$};
    \draw (0,-0.1) node[below left] {$0$};
    \draw (0.4,-0.1) node[below] {$F_1^*$};
    \draw[dotted] (0.4,-0.05) -- (0.4,1.1);
    \draw[color=blue] (0.4,1.1) node {$\bullet$};
    \draw (2.45,-0.1) node[below] {$F^*$};
    \draw[dotted] (2.45,-0.05) -- (2.45,3.18);
    \draw[color=blue] (2.45,3.17) node {$\bullet$};
\end{tikzpicture} 
\caption{Schematic representations of the function $\Gamma(\phi_0)$ and its intersections with the affine function (dotted line) defined by the right hand side of \eqref{rel_stat3} for small $\gamma$ (left) and for larger $\gamma$ (right). In the latter case there are two intersections $F_1^*<F^*$; moreover, since $\Gamma$ is increasing with respect to $\gamma$, we see clearly that the larger $\gamma$ is, the smaller $F_1^*$ and the larger $F^*$ are.}\label{fig:Gamma}
\end{center}
\end{figure}
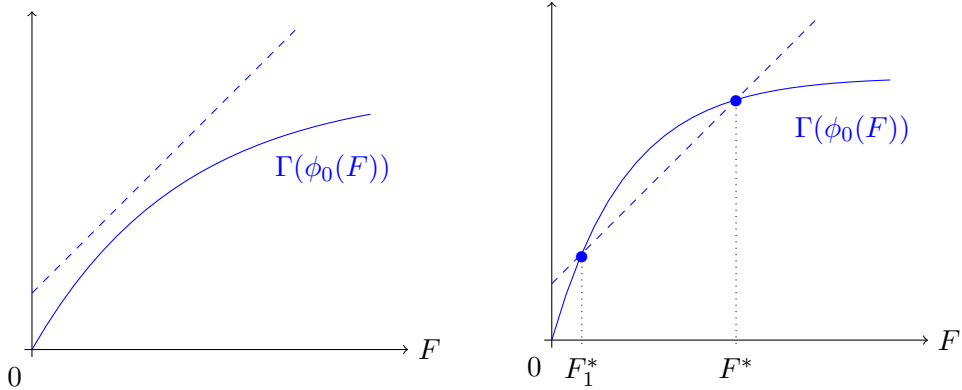

Setting $m=\dfrac{bF^*}{bF^*+K(\mu_E+\nu_E)}$, which is equivalent to $\dfrac{bF^*}{K} = \dfrac{(\mu_E+\nu_E) m}{1-m}$, and using the notation in \eqref{def:Nzeta}, we have $\phi_0(F^*) = \dfrac{m}{\zeta\gamma}$, and equation \eqref{rel_stat2} rewrites
\begin{equation}\label{eqapp:phi}
    \mathcal{N} (1-e^{-m/\zeta}) (1 - m) = 1.
\end{equation}

Let us denote $\varphi(m) := \mathcal{N} (1-e^{- m/\zeta}) \left(1 - m\right)$. 
We may verify easily that $\varphi(0) = 0$, $\varphi(1) = 0$ and $\varphi$ is concave on $(0,1)$.
As a consequence $\varphi$ admits a unique maximum on $(0,1)$ which is reached at point $m_0\in (0,1)$. The point $m_0$ is characterized by $\varphi'(m_0) = 0$, which is equivalent to 
\begin{equation}
    \label{eq:defm0}
1- m_0 = \zeta (e^{m_0/\zeta}-1).
\end{equation}
Notice that since the left-hand side is decreasing and the right-hand side increasing with respect to $m_0$, we deduce that for $m\in (0,1)$, we have 
\begin{equation}\label{eq:mm0}
m<m_0 \text{ if and only if } \zeta (e^{m/\zeta}-1) < 1- m.
\end{equation}
Since $\varphi$ is continuous and concave on $(0,1)$ and nonpositive elsewhere, equation \eqref{eqapp:phi} has a unique solution if and only if $\varphi(m_0) = 1$, and it has two positive solutions $m^*_-$ and $m^*_+$ such that $0<m^*_-<m_0<m^*_+$ if and only if $\varphi(m_0) > 1$ which is equivalent to
$$
\mathcal{N} (1-e^{-m_0/\zeta}) (1-m_0) > 1.
$$
With the relation \eqref{eq:defm0}, it implies
\begin{equation}
    \label{estimm0}
    \zeta \mathcal{N} (1-e^{-m_0/\zeta})(e^{m_0.\zeta}-1) > 1.
\end{equation}
Then,
$$
\zeta \mathcal{N} (e^{m_0/\zeta})^2 - (2\zeta \mathcal{N}+1) e^{m_0/\zeta} + \zeta \mathcal{N} >0.
$$
Solving this second order polynomial in $e^{m_0/\zeta}$ inequality, the latter inequality is equivalent (since $m_0>0$) to 
$$
e^{m_0/\zeta} > \frac{2\zeta \mathcal{N} + 1 + \sqrt{4\zeta \mathcal{N} + 1}}{2\zeta \mathcal{N}}.
$$
Let us denote 
$$
m^* := \zeta\ln \left(\frac{2\zeta \mathcal{N} + 1 + \sqrt{4\zeta \mathcal{N} + 1}}{2\zeta \mathcal{N}} \right).
$$
From \eqref{eq:mm0}, we deduce that $m^*<m_0$ if and only if
$$
\zeta (e^{m^*/\zeta}-1) < 1 - m^*.
$$
We conclude that there exist exactly two positive steady states if and only if
\begin{equation}
    \label{cond:stat}
\frac{1+\sqrt{4\zeta\mathcal{N}+1}}{2\mathcal{N}} < 1-\zeta \ln \left(\frac{2\zeta \mathcal{N} + 1 + \sqrt{4\zeta \mathcal{N} + 1}}{2\zeta \mathcal{N}} \right).
\end{equation}

We may verify that given $\mathcal{N}>1$, the function $\zeta \mapsto \frac{1+\sqrt{4\zeta \mathcal{N}+1}}{2\mathcal{N}}$ is increasing whereas $\zeta\mapsto 1-\zeta \ln (\frac{2\zeta \mathcal{N} + 1 + \sqrt{4\zeta \mathcal{N} + 1}}{2\zeta \mathcal{N}})$ is decreasing. 
Moreover, for $\zeta=0$ condition \eqref{cond:stat} reads $\mathcal{N}>1$ which is assumed to be satisfied.
Therefore defining $\zeta_c$ as in the statement of Lemma \ref{lem:stat},
we have that condition \eqref{cond:stat} is equivalent to $\zeta < \zeta_c$, or equivalently $\gamma > \gamma_c$ where 
$\gamma_c = \dfrac{\mu_M}{(1-\rho) \nu_E \zeta_c K}$.
Moreover, since $\Gamma$ is increasing with respect to $\gamma$, we deduce from \eqref{rel_stat3} that $F^*$ is increasing with respect to $\gamma$ and $F_1^*$ is decreasing with respect to $\gamma$.

Finally, there exists exactly one positive steady state if the inequality in \eqref{cond:stat} is an equality, i.e. $\zeta=\zeta_c$ or equivalently $\gamma=\gamma_c$.
In this case, we have on $(0,F^*)$ 
\begin{equation}\label{rel_stat4}
\Gamma(\phi_0(F)) < \frac{\mu_F F}{\rho \nu_E K}+ \frac{1}{\mathcal{N}}.
\end{equation}

We have studied the condition of existence of the three equilibria in the bistable case. The analysis of their stability is similar to the one in \cite[Lemma 3]{STR19} and is not reproduced here.

\bibliographystyle{plain} 
{\footnotesize
\bibliography{SIT}}

\begin{thebibliography}{10}

\bibitem{AgboAlmeidaCoron2}
Kala {Agbo bidi}, Lu{\'i}s Almeida, and Jean-Michel Coron.
\newblock {Feedback stabilization for a spatial-dependent Sterile Insect
  Technique model with Allee Effect}.
\newblock working paper or preprint, February 2025.

\bibitem{bidiJOTA2025}
Kala {Agbo bidi}, Luis Almeida, and Jean-Michel Coron.
\newblock Global stabilization of a sterile insect technique model by feedback
  laws.
\newblock {\em Journal of Optimization Theory and Applications}, 204(30), 2025.

\bibitem{BIDI2025}
Kala {Agbo bidi}, Jean-Michel Coron, Amaury Hayat, and Nathan Lichtlé.
\newblock A novel approach to feedback control with deep reinforcement
  learning.
\newblock {\em Systems \& Control Letters}, 202:106102, 2025.

\bibitem{ADPV_MBE2019}
Luis Almeida, Michel Duprez, Yannick Privat, and Nicolas Vauchelet.
\newblock Mosquito population control strategies for fighting against
  arboviruses.
\newblock {\em Mathematical Biosciences and Engineering}, 16(6):6274--6297,
  2019.

\bibitem{ADPV2022}
Luis Almeida, Michel Duprez, Yannick Privat, and Nicolas Vauchelet.
\newblock Optimal control strategies for the sterile mosquitoes technique.
\newblock {\em Journal of Differential Equations}, 311:229--266, 2022.

\bibitem{ALM1}
Luis Almeida, Jorge Estrada, and Nicolas Vauchelet.
\newblock {The sterile insect technique used as a barrier control against
  reinfestation}.
\newblock In {\em {Optimization and Control for Partial Differential
  Equations}}, volume~29 of {\em Radon Series on Computational and Applied
  Mathematics}, pages 91--112. {De Gruyter}, March 2022.

\bibitem{AlmLecNadPriv}
Luis Almeida, Alexis L\'{e}culier, Gr\'{e}goire Nadin, and Yannick Privat.
\newblock Optimal control of bistable traveling waves: Looking for the best
  spatial distribution of a killing action to block a pest invasion.
\newblock {\em SIAM Journal on Control and Optimization}, 62(2):1291--1315,
  2024.

\bibitem{ALM3}
Luis Almeida, Alexis L{\'e}culier, and Nicolas Vauchelet.
\newblock {Analysis of the ''Rolling carpet'' strategy to eradicate an invasive
  species}.
\newblock {\em {SIAM Journal on Mathematical Analysis}}, 55(1):275--309,
  February 2023.

\bibitem{ANG}
Roumen Anguelov, Yves Dumont, and Ivric Valaire~Yatat Djeumen.
\newblock On the use of {Traveling} {Waves} for {Pest}/{Vector} elimination
  using the {Sterile} {Insect} {Technique}, October 2020.
\newblock arXiv:2010.00861 [math].

\bibitem{Bellini2013}
Romeo Bellini, Anna Medici, Arianna Puggioli, Fabricio Balestrino, and Maurizio
  Carrieri.
\newblock Pilot field trials with aedes albopictus irradiated sterile males in
  italian urban areas.
\newblock {\em Journal of Medical Entomology}, 50(2):317--325, 03 2013.

\bibitem{bliman2024}
Pierre-Alexandre Bliman.
\newblock Basic offspring number and robust feedback design for the biological
  control of vectors by sterile insect release technique, 2024.

\bibitem{BLI}
Pierre-Alexandre Bliman, Daiver Cardona-Salgado, Yves Dumont, and Olga
  Vasilieva.
\newblock Implementation of {Control} {Strategies} for {Sterile} {Insect}
  {Techniques}.
\newblock {\em Mathematical Biosciences}, 314:43--60, August 2019.
\newblock Publisher: Elsevier.

\bibitem{Bouyer2024}
Jérémy Bouyer.
\newblock Current status of the sterile insect technique for the suppression of
  mosquito populations on a global scale.
\newblock {\em Infectious Diseases of Poverty}, 13(68), 2024.

\bibitem{bressan}
Alberto Bressan, Maria~Teresa Chiri, and Najmeh Salehi.
\newblock On the optimal control of propagation fronts.
\newblock {\em Mathematical Models and Methods in Applied Sciences},
  32(06):1109--1140, 2022.

\bibitem{CAP}
Beniamino Caputo, Riccardo Moretti, Mattia Manica, Paola Serini, Elena
  Lampazzi, Marco Bonanni, Giulia Fabbri, Verena Pichler, Alessandra della
  Torre, and Maurizio Calvitti.
\newblock A bacterium against the tiger: preliminary evidence of fertility
  reduction after release of aedes albopictus males with manipulated wolbachia
  infection in an italian urban area.
\newblock {\em Pest Management Science}, 76, 10 2019.

\bibitem{cristofaro2024}
Andrea Cristofaro and Luca Rossi.
\newblock Backstepping control for the sterile mosquitoes technique:
  stabilization of extinction equilibrium, 2024.

\bibitem{Enkerlin2017}
Walther~R. Enkerlin~et al.
\newblock The moscamed regional programme: review of a success story of
  area-wide sterile insect technique application.
\newblock {\em Entomologia Experimentalis et Applicata}, 164(3):188--203, 2017.

\bibitem{FAN}
Jian Fang and Xiao-Qiang Zhao.
\newblock Monotone {Wavefronts} for {Partially} {Degenerate}
  {Reaction}-{Diffusion} {Systems}.
\newblock {\em J Dyn Diff Equat}, 21(4):663--680, December 2009.

\bibitem{GAT}
Rene Gato~Armas, Zulema Menéndez, Enrique Prieto, Rafael Argilés, Misladys
  Rodríguez, Waldemar Baldoquín~Rodríguez, Yisel Hernández~Barrios, Dennis
  Pérez~Chacón, Jorge Anaya, Ilario Fuentes, Claudia Lorenzo, Keren
  González, Yudaisi Campo, and Jérémy Bouyer.
\newblock Sterile insect technique: Successful suppression of an aedes aegypti
  field population in cuba.
\newblock {\em Insects}, 12:469, 05 2021.

\bibitem{Girardin}
L{\'e}o Girardin.
\newblock The effect of random dispersal on competitive exclusion - a review.
\newblock {\em Math. Biosci.}, 318:8, 2019.
\newblock Id/No 108271.

\bibitem{Freefem}
Frédéric Hecht.
\newblock New development in freefem++.
\newblock {\em J. Numer. Math.}, 20(3-4):251--265, 2012.

\bibitem{LEC2023}
Alexis Leculier and Nga Nguyen.
\newblock A control strategy for the sterile insect technique using
  exponentially decreasing releases to avoid the hair-trigger effect.
\newblock {\em Math. Model. Nat. Phenom.}, 18:25, 2023.

\bibitem{ABPV2025}
{Luis Almeida}, {Jesús Bellver-Arnau}, {Gwenaël Peltier}, and {Nicolas
  Vauchelet}.
\newblock Optimal strategies for wolbachia mosquito replacement technique:
  influence of the carrying capacity on spatial releases.
\newblock {\em ESAIM: COCV}, 31:57, 2025.

\bibitem{Maiga}
Hamidou Maïga, Wadaka Mamai, Nanwintoum~Séverin Bimbilé~Somda, Anna Konczal,
  Thomas Wallner, Gustavo~Salvador Herranz, Rafael~Argiles Herrero, Hanano
  Yamada, and Jeremy Bouyer.
\newblock Reducing the cost and assessing the performance of a novel adult
  mass-rearing cage for the dengue, chikungunya, yellow fever and zika vector,
  aedes aegypti (linnaeus).
\newblock {\em PLOS Neglected Tropical Diseases}, 13(9):1--21, 09 2019.

\bibitem{MULTERER2019}
Lea Multerer, Thomas Smith, and Nakul Chitnis.
\newblock Modeling the impact of sterile males on an aedes aegypti population
  with optimal control.
\newblock {\em Mathematical Biosciences}, 311:91--102, 2019.

\bibitem{PER}
Benoît Perthame.
\newblock {\em Parabolic {Equations} in {Biology}: {Growth}, reaction, movement
  and diffusion}.
\newblock Springer, September 2015.
\newblock Google-Books-ID: 0pOKCgAAQBAJ.

\bibitem{SEI}
S.~Seirin~Lee, Ruth~E. Baker, Eamonn~A. Gaffney, and Steven~M. White.
\newblock Modelling {Aedes} aegypti mosquito control via transgenic and sterile
  insect techniques: {Endemics} and emerging outbreaks.
\newblock {\em Journal of Theoretical Biology}, 331:78--90, August 2013.

\bibitem{STR19}
Martin Strugarek, Hervé Bossin, and Yves Dumont.
\newblock On the use of the sterile insect release technique to reduce or
  eliminate mosquito populations.
\newblock {\em Applied Mathematical Modelling}, 68:443--470, 2019.

\bibitem{THOME2010}
Roberto~C.A. Thomé, Hyun~Mo Yang, and Lourdes Esteva.
\newblock Optimal control of aedes aegypti mosquitoes by the sterile insect
  technique and insecticide.
\newblock {\em Mathematical Biosciences}, 223(1):12--23, 2010.

\bibitem{VOL}
Aizik Volpert, Vitaly Volpert, and Vladimir Volpert.
\newblock {\em Traveling {Wave} {Solutions} of {Parabolic} {Systems}}, volume
  140 of {\em Translations of {Mathematical} {Monographs}}.
\newblock American Mathematical Society, October 1994.
\newblock ISSN: 0065-9282, 2472-5137.

\bibitem{ZHE}
Xiaoying Zheng, Dongjing Zhang, Yongjun Li, Cui Yang, Yu~Wu, Xiao Liang,
  Yongkang Liang, Xiaoling Pan, Linchao Hu, Qiang Sun, Xiaohua Wang, Yingyang
  Wei, Jian Zhu, Wei Qian, Ziqiang Yan, Andrew Parker, Jeremie Gilles, Kostas
  Bourtzis, Jérémy Bouyer, and Zhiyong Xi.
\newblock Incompatible and sterile insect techniques combined eliminate
  mosquitoes.
\newblock {\em Nature}, 572:1, 08 2019.

\end{thebibliography}
\end{document}